\documentclass[11pt]{amsart}
\usepackage{amsmath}
\usepackage{amssymb}
\usepackage{amscd}
\usepackage{color}

\def\NZQ{\mathbb}               % the font for N,Z,Q,R,C
\def\NN{{\NZQ N}}

\def\ZZ{{\NZQ Z}}
\def\RR{{\NZQ R}}

\newtheorem{Theorem}{Theorem}[section]
\newtheorem{Lemma}[Theorem]{Lemma}
\newtheorem{Corollary}[Theorem]{Corollary}
\newtheorem{Proposition}[Theorem]{Proposition}
\newtheorem{Remark}[Theorem]{Remark}

\newtheorem{Definition}[Theorem]{Definition}

\let\epsilon\varepsilon
\let\phi=\varphi
\let\kappa=\varkappa

\begin{document}

%\title{Ramification and distance in   Artin-Schreier extensions (Preliminary version)}
\title{Erratic birational behavior of mappings in positive characteristic}
\author{Steven Dale Cutkosky}
\thanks{Partially supported by NSF grant DMS-2054394}

\address{Steven Dale Cutkosky, Department of Mathematics,
University of Missouri, Columbia, MO 65211, USA}
\email{cutkoskys@missouri.edu}

%\subjclass[2000]{primary 14B05;  secondary 14B25, 13A18} 

\begin{abstract} Birational properites of generically finite morphisms $X\rightarrow Y$ of algebraic varieties can be understood locally by a valuation of the function field of $X$.
In finite extensions of algebraic local rings in characteristic zero algebraic function fields which are dominated by a valuation there are   nice monomial forms of the mapping after blowing up enough, which reflect classical invariants of the valuation. Further, these forms are stable upon suitable further blowing up. In positive characteristic algebraic function fields it is not always possible to find a monomial form after blowing up along a valuation, even in dimension two. In dimension two and positive characteristic, after enough blowing up, there are stable forms of the mapping which hold upon suitable sequences of blowing. We give examples showing that even within these stable forms, the forms can vary dramatically (erratically) upon further blowing up.  We construct these examples in defect Artin-Schreier extensions which can have any prescribed distance. As an application, we give an example of an extension of two dimensional regular local rings in a tower of two independent defect Artin-Schreier extensions for which strong local monomialization does not hold. 
\end{abstract}

\maketitle 

\section{Introduction}

Suppose that $\phi:X\rightarrow Y$ is a morphism of algebraic varieties over a field $k$. We would like to find resolutions of singularities $X_1$ of $X$ and $Y_1$ of $Y$ such that there is an induced morphism $X_1\rightarrow Y_1$ which has the simplest possible local structure. In \cite[Problem 6.2.1]{AKMW}, it is asked if, with the assumption that $k$ has characteristic zero,  there always exists such an $X_1\rightarrow Y_1$ which is toroidal (the problem of toroidalization). Certainly no simpler global structure can always be found.
This toroidal form can be found if the varieties have characteristic zero and are of dimension $\le 3$ (\cite{C4}, \cite{C5} and a simplified proof in \cite{C3}).
Stronger local forms than toroidalization are true in complete generality in characteristic zero.  To formulate this result, we use valuations.

The use of valuations converts the study  of birational properties of morphisms into a problem in local commutative algebra. This approach was initiated and   developed by Zariski. Suppose that 
$\phi:Y\rightarrow X$ is a dominant morphism of projective (or proper) $k$-varieties and we have a commutative diagram of morphisms
$$
\begin{array}{ccc}
Y_1&\rightarrow & X_1\\
\downarrow && \downarrow\\
Y&\rightarrow & X
\end{array}
$$
where $Y_1\rightarrow Y$ and $X_1\rightarrow X$ are  birational and projective (or proper). A valuation $\omega$ of the function field $k(Y)$ (which is trivial on $k$)  restricts to a valuation $\nu$ on $k(X)$. The valuations have centers at points $y$ and $x$ on $Y$ and $X$ respectively and have centers $y_1$ and $x_1$ on $Y_1$ and $X_1$ respectively.  The behavior of the morphisms in the diagram near the centers of the valuations is determined completely by the commutative diagram
$$
\begin{array}{ccc}
R_1&\rightarrow & S_1\\
\uparrow && \uparrow\\
R&\rightarrow & S
\end{array}
$$
where $R,S,R_1,S_1$ are the respective local rings $\mathcal O_{X,x}, \mathcal O_{Y,y}, \mathcal O_{X,x_1},\mathcal O_{Y,y_1}$.
Thus, we are able to formulate locally problems in birational geometry in terms of valuation theory and local commutative algebra. 

A classical example of this approach is the problem of local uniformization as formulated by Zariski. Suppose that $R$ is a local domain which is essentially of finite type  over a field $k$ with quotient field $K$. Let  $\nu$ be a valuation of $K$ which dominates $R$. A local uniformization of $R$ along $\nu$ is a birational extension $R\rightarrow R_1$ where $R_1$ is a regular local ring which is essentially of finite type over $R$ and is dominated by $\nu$. 
Zariski proved local uniformization in all dimensions and characteristic zero in \cite{Z2} and
established resolution of singularities in characteristic zero from local uniformization in dimension three \cite{Z3}. Hironaka later proved resolution of singularities in all dimensions and characteristic zero in \cite{H1}, using a different method.  All current proofs of resolution of singularities of 3-folds in positive characteristic are accomplished by first proving local uniformization (\cite{Ab3},  A simplification of this proof in \cite{C6}, Cossart and Piltant \cite{CoP}). 
As of this time, the existence of resolution of singularities in dimension $\ge 4$ is unknown in positive characteristic.
The problem of local uniformization (and resolution of singularities) is closely related to the problem of finding good local forms of mappings along a  valuation, as the construction of a resolution of singularities generally starts with a generically finite projection onto a nonsingular variety.

In characteristic zero, there is a very nice local form for morphisms, called local monomialization. This result is a little stronger than what comes immediately from the assumption that toroidalization is possible.  If $R\rightarrow S$ is an extension of local rings such that the maximal ideal of $S$ contracts to the maximal ideal of $R$ then we say that $S$ dominates $R$. If $S$ is dominated by the valuation ring $\mathcal O_{\omega}$ of a valuation $\omega$ we say that $\omega$ dominates $S$.

\begin{Theorem}\label{locmon}(local monomialization)(\cite{C}, \cite{C5}) Suppose that $k$ is a field of characteristic zero and $R\rightarrow S$ is an extension of regular local rings such that $R$ and $S$ are essentially of finite type over $k$ and $\omega$ is a valuation of the quotient field of $S$ which dominates $S$ and $S$ dominates $R$. Then there is a 
commutative diagram
$$
\begin{array}{ccc}
R_1&\rightarrow & S_1\\
\uparrow&&\uparrow\\
R&\rightarrow &S
\end{array}
$$
such that  $\omega$ dominates $S_1$, $S_1$ dominates $R_1$ and the vertical arrows are products of monoidal transforms; that is, these arrows are factored by the local rings of blowups of prime ideals  whose quotients are regular local rings. In particular, $R_1$ and $S_1$ are regular local rings. Further, $R_1\rightarrow S_1$ has a locally monomial form; that is, there exist regular parameters $u_1,\ldots,u_m$ in $R_1$ and $x_1,\ldots,x_n$ in $S_1$, an $m\times n$ matrix $A=(a_{ij})$ with integral coefficients such that 
$\mbox{rank}(A)=m$ and units $\delta_i\in S_1$ such that
$$
u_i=\delta_i\prod_{j=1}^nx_j^{a_{ij}}
$$
for $1\le i\le m$.
\end{Theorem}
The difficulty in the proof is to obtain the condition that $\mbox{rank}(A)=m$. To do this, it is necessary to blow up above both $R$ and $S$.

In the case when the extension of quotient fields $K\rightarrow L$ of the extension $R\rightarrow S$ is a finite extension and $k$ has characteristic zero, it is possible to find a local monomialization such that the structure of the matrix of coefficients recovers classical invariants of the extension of valuations in $K\rightarrow L$, and this form holds stably along suitable sequences of birational morphisms which generate the respective valuation rings. This form is called strong local uniformization. It is established for rank 1 valuations in \cite{C} and for general valuations in \cite{CP}. The case which has the simplest form and will be of interest to us in this paper is when the valuation has rational rank 1. In this case, if $R_1\rightarrow S_1$ is a strong local monomialization, then there exist regular parameters $u_1,\ldots, u_m$ in $R_1$ and $v_1,\ldots,v_m$ in $S_1$ , a positive integer $a$ and a unit $\delta\in S_1$ such that
\begin{equation}\label{N20}
u_1=\delta v_1^a, u_2=v_2,\ldots,u_m=v_m.
\end{equation}

The stable forms of mappings in positive characteristic and dimension $\ge 2$ are much more complicated. For instance, local monomialization does not always hold. An example is given in \cite{C2} where $R\rightarrow S$ are local rings of points on nonsingular algebraic surfaces over an algebraically closed field $k$ of positive characteristic $p$ and $k(X)\rightarrow k(Y)$ is finite and separable. 

This leads to the question of determining the best local forms that are possible along a valuation for a generically finite morphism of surfaces in positive characteristic. In this paper, we will consider this situation,  assuming that $k$ is algebraically closed of positive characteristic $p$.

The obstruction to local monomialization is the defect.   The defect $\delta(\omega/\nu)$, which is a power of the residue characteristic $p$ of $\mathcal O_{\omega}$,  is defined and its basic properties developed in \cite[Chapter VI, Section 11]{ZS2}, \cite{K0}, \cite[Section 7.1]{CP}. The defect is discussed in Subsection \ref{Subsecnot}. We have the following theorem, showing that the defect is the only obstruction to strong local monomialization for maps of surfaces.

\begin{Theorem}\label{nondef}(\cite[Theorem 7.35]{CP}) Suppose that $K\rightarrow L$ is a finite, separable extension of algebraic function fields over an algebraically closed field $k$ of characteristic $p>0$, $R\rightarrow S$ is an extension of local domains such that $R$ and $S$ are essentially of finite type over $k$ and the quotient fields of $R$ and $S$ are $K$ and $L$ respectively such that $S$ dominates $R$. Suppose that $\omega$ is valuation of $L$ which dominates $S$. Let $\nu$ be the restriction of $\omega$ to $K$. Suppose that the extension is defectless ($\delta(\omega/\nu)=1$). Then the conclusions of Theorem \ref{locmon} hold. In particular, $R\rightarrow S$ has a local monomialization (and a strong local monomialization) along $\omega$. 
\end{Theorem}

Suppose that 
$K\rightarrow L$ is a Galois extension of fields of characteristic $p>0$  and $\omega$ is a valuation of $L$, $\nu$ is the restriction of $\omega$ to $K$. Then there is a classical tower of fields (\cite[page 171]{End})
$$
K\rightarrow K^s\rightarrow K^i\rightarrow K^v\rightarrow L.
$$
where $K^s$ is the splitting field, $K^i$ is the inertia field, $K^v$ is the ramification field and the extension $K\rightarrow K^v$ has no defect. Thus the essential difficulty comes from the extension from  $K^v$ to $L$ which could have defect. The extension $K^v\rightarrow L$ is a tower of Artin-Schreier extensions, so the Artin-Schreier extension is of fundamental importance in this theory.  

Kuhlmann has extensively studied defect in Artin-Schreier extensions in \cite{Ku}. He separated these extensions into  dependent and independent defect Artin-Schreier extensions. This definition is reproduced in Subsection \ref{Galois}.
Kuhlmann  also defined an invariant called the distance to distinguish the natures of Artin-Schreier extensions. This definition is  given in Subsections \ref{SecDist} and \ref{Galois}.

We now specialize to the case of  a finite separable  extension $K \rightarrow L$ of two   dimensional algebraic function fields  over an algebraically closed field $k$ of characteristic $p>0$, and suppose that $\omega$ is a valuation of $L$ which is trivial on $k$ and $\nu$ is the restriction of $\omega$ to $K$.  If $L/K$ has defect then $\omega$ must have  rational rank 1 and be nondiscrete. We will assume that $\omega$ has rational rank 1 and is nondiscrete for the remainder of the introduction.

 With these restrictions, the distance $\delta$ of an Artin-Schreier extension 
is $\le$ $0^{-}$ when the extension has  defect. If it is a defect extension with $\delta=0^{-}$  then it is  an independent defect extension. If it is a defect extension and the distance is less than $0^{-}$ then the extension is a dependent defect extension.

A quadratic transform along a valuation is the center of the valuation at the blow up of a maximal ideal of a regular local ring. There is the sequence of quadratic transforms along $\nu$ and $\omega$
\begin{equation}\label{In2}
R\rightarrow R_1\rightarrow R_2\rightarrow \cdots\mbox{ and }S\rightarrow S_1\rightarrow S_2\rightarrow\cdots.
\end{equation}
We have that $\cup_{i=1}^{\infty}R_i=\mathcal O_{\nu}$, the valuation ring of $\nu$, and $\cup_{i=1}^{\infty}S_i=\mathcal O_{\omega}$, the valuation ring of $\omega$.
 These sequences can be factored by standard quadratic transform sequences (defined in Section \ref{Sec2Ext}).  It is shown in \cite{CP} that given positive integers $r_0$ and $s_0$, there exists $r\ge r_0$ and $s\ge s_0$ such that $R_r\rightarrow S_s$ has the following  form:
\begin{equation}\label{In1}
u=\delta x^a, v=x^b(y^d\gamma+x\Omega)
\end{equation}
where $u,v$ are regular parameters in $R_r$, $x,y$ are regular parameters in $S_s$, $\gamma$ and $\tau$ are units in $S_s$,
$\Omega\in S_s$, $a$ and $d$ are positive integers and $b$ is a non negative integer. If we choose $r_0$  sufficiently large, then
we have  that the complexity $ad$ of the extension $R_r\rightarrow S_s$  is a constant which depends on the extension of valuations, which we call the stable complexity of (\ref{In2}). When $R_r\rightarrow S_s$ has this stable complexity, we  call the forms (\ref{In1}) stable forms.

  The strongly monomial form is the case when $b=0$ and $d=1$; that is, after making a change of variables in $y$,
  $$
  u=\delta x^a, v=y.
  $$

As we observed earlier (Theorem \ref{nondef}) if the extension $K\rightarrow L$ has no defect, then the stable form is the strongly monomial form. If there is defect, then it is possible for the $a$ and $d$ in stable forms along a valuation to vary wildly, even though their product $ad$ is fixed by the extension, as we will see in Theorem \ref{Example1}.
An interesting open question is if we can always find stable local forms (\ref{In1}) with $b=0$.

We make an extensive study of the local forms  which can occur in an Artin-Schreier extension under a sequence of quadratic transforms in Section \ref{SecCalc}.
In an Artin-Schreier extension, the stable complexity is either 1 or $p$. If the stable complexity is 1, then a stable form $R_r\rightarrow S_s$ is unramified. If the stable complexity is $p$, then a stable form $R_r\rightarrow S_s$ is either of type 1 or of type 2, as defined below, and the type can vary within the sequences (\ref{In2}). The two types are defined as follows. There are regular parameters $u_r,v_r$ in $R_r$ and $x_s, y_s$ in $S_s$ such that if $R_r\rightarrow S_s$ has type 1, then 
\begin{equation}\label{In3}
u_r=x_s, v_r=y_s^p\gamma+x_s\Sigma
\end{equation}
 where $\gamma$ is a unit in $S_s$ and $\Sigma\in S_s$.
If $R_r\rightarrow S_s$ is of type 2, then  
\begin{equation}\label{In4}
u_r=\gamma x_s^p, v_r=y_s
\end{equation}
 where $\gamma,\tau$ are units in $S_s$ and $\Omega\in S_s$. The Artin-Schreier extension $L/K$ has no defect if and only if the stable forms $R_r\rightarrow S_s$ are   of type 2 for $r\gg 0$ (Proposition  \ref{PropositionSB}). Observe that type 2 is the condition of being strongly monomial.
 
 In Theorem \ref{Example1}, it is shown that we can construct  defect Artin-Schreier extensions with any prescribed (nonpositive) distance and any  prescribed switching between types.  Since we are constructing defect extensions, we must impose the condition that the stable forms are not 
eventually always of type 2.

The construction is such that if 
\begin{equation}\label{In5}
R=R_0\rightarrow R_1\rightarrow R_2\rightarrow\cdots\mbox{ and }S=S_0\rightarrow S_1\rightarrow S_2\rightarrow \cdots
\end{equation}
are the sequences of standard quadratic transform sequences along the valuation, then $R_i\rightarrow S_i$ is a stable form for all $i$, and if $\Phi:\NN\rightarrow \{1,2\}$ is any function (the prescribed switching) which is not equal to 2 for all sufficiently large integers, then $R_i\rightarrow S_i$ is of type $\Phi(i)$ for all $i$.

We use a formula derived in Proposition \ref{Prop100} to compute distances from sequences (\ref{In2}) in Artin-Schreier extensions by Piltant and Kuhlmann \cite{KP}. For the reader's convenience, we give a proof of this formula in the appendix to this paper. 

We analyze in Section \ref{Example} an example in \cite{CP}, showing failure of strong local monomialization.    It is a tower of two defect Artin-Schreier extensions, each  of the type of Theorem \ref{Example1}. The first extension is of type 1 for even integers and of type 2 for odd integers. The second extension is of type 2 for even integers and of type 1 for odd integers. The composite gives a sequence of extensions of regular local  rings $R_i\rightarrow S_i$, where $R_i$ has regular parameters $u_i,v_i$ and $S_i$ has regular parameters $x_i,y_i$ such that the stable form is 
\begin{equation}\label{eqN25}
u_i=\gamma x_i^{p}, v_i=y_i^p\tau+x_i\Omega
\end{equation}
for all $i$.

Using the formula of Proposition \ref{Prop100}, we compute the distances of these two Artin-Schreier extensions,  concluding  that both extensions are  dependent. We show that the first Artin-Schreier extension has distance $(-\frac{p^4-2}{p^4-1})^-$ and the second Artin-Schreier extension has distance $(-\frac{cp^3+(c-1)p^2+cp+c}{p^4-1})^-$, where $c$ is a number occuring in the equation defining the second extension. The first of these distances was computed in \cite{EG} by a different method.

%By Theorem \ref{Example2}, we can construct a similar example which is also a tower of two Artin-Schreier extensions with the same types of switching and so that we also have that $\psi(i)=1$ for all $i$, giving the stable forms (\ref{eqN25})  and such that both of the Artin-Schreier extensions are independent.

In Theorem \ref{Example3}, we show that this type of example can be constructed in a tower of two independent defect extensions. In particular, this gives an example in a tower of two independent defect Artin-Schreier extensions for which strong local monomialization does not hold.

Suppose that  $K\rightarrow L$ is a finite extension of  fields of positive characteristic and $\omega$ is a valuation of $L$  with restriction $\nu$ to $K$ such that $\omega$ is the unique extension of $\nu$ to $L$. It is known that there is no defect in the extension if and only if there is a finite generating sequence in $L$ for the valuation $\omega$ over $K$ (\cite{Va}, \cite{NS}). The calculation of generating sequences for extensions of Noetherian local rings which are dominated by a valuation is extremely difficult. This has been accomplished for two  dimensional  regular local rings  in \cite{Sp} and \cite{CV} and for many hypersurface singularities above a regular local ring of arbitrary dimension in \cite{CMT}.

The nature of a generating sequence in an extension of  $S$ over $R$ determines the nature of the mappings in the stable forms. It is shown in \cite[Theorem 1]{C8} that if $R\rightarrow S$ is an extension of two dimensional excellent regular local rings whose quotient fields give a finite extension $K\rightarrow L$ and $\omega$ is a valuation of $L$ which dominates $S$ then the extension is without defect if and only if  there exist sequences of quadratic transform $R\rightarrow R_1$ and $S\rightarrow S_1$ along $\nu$ 
such that $\omega$ has  a finite generating sequence in $S_1$ over $R_1$.
This shows us that we can expect good stable forms (as do hold by Theorem \ref{nondef}) if there is no defect, but not otherwise.
 
 I thank  Franz-Viktor Kuhlmann  and Olivier Piltant  for  
telling me about the beautiful formulas from \cite{KP}, comparing distance and ramification cuts in an Artin-Schreier extension,
and sharing their manuscript \cite{KP} with me. For the reader's convenience, I give proofs of these formulas in the appendix. 

%A formula is given in \cite{KP} showing that the ramification  cut can be computed from the Jacobian ideal $J(S_i/R_i)$ of the extensions occurring in  a diagram (\ref{eq23*1}). 
%In \cite{KP}, this formula is proven in the case when $\nu K$ is not $p$-divisible and is applied in Theorem 16 \cite{KP} in this case of Artin-Shreier extensions where the value groups are not $p$-divisible. Further, they show that in this case, the sequence
%$R_r\rightarrow R_{r+1}\rightarrow \cdots$ for large $r$ in the diagram (\ref{eq23*1}) is uniquely determined by $\nu$, 

In \cite{KP}, Kuhlmann and Piltant compute for Artin-Schreier extensions of two dimensional algebraic function fields over an algebraically closed field,  the possible distances which occur for valuations, with the assumption  that the value group is not $p$ divisible, and conclude that this number is finite. In the later paper \cite{BK}, it is shown that quite generally there are distance bounds for Artin-Schreier extensions, and in particular in the case considered in this paper of two dimensional algebraic function fields over an algebraically closed field, the bound is 4.

\section{Preliminaries}\label{SecPre}
\subsection{Some notation}\label{Subsecnot}  Let $K$ be a field with a valuation $\nu$.  The valuation ring of $\nu$ will be donoted by $\mathcal O_{\nu}$, $\nu K$ will denote the value group of $\nu$ and $K\nu$ will denote the residue field of $\mathcal O_{\nu}$. 

The maximal ideal of a local ring $A$ will be denoted by $m_A$. If $A\rightarrow B$ is an extension (inclusion) of local rings such that $m_B\cap A=m_A$ we will say that $B$ dominates $A$. If a valuation ring $\mathcal O_{\nu}$ dominates $A$ we will say that the valuation $\nu$ dominates $A$.

Suppose that $K$ is an algebraic function field over a field $k$. An algebraic local ring $A$ of $K$ is a local domain which is a localization of a finite type $k$-algebra whose quotient field is $K$. A $k$-valuation of $K$ is a valuation of $K$ which is trivial on $k$.

Suppose that $K\rightarrow L$ is a finite algebraic extension of fields, $\nu$ is a valuation of $K$  and $\omega$ is an extension of $\nu$ to $L$. Then the reduced ramification index of the extension is $e(\omega/\nu)=[\omega L:\nu K]$ and the residue degree of the extension is $f(\omega/\nu)=[L\omega:K\nu]$.

 The defect $\delta(\omega/\nu)$, which is a power of the residue characteristic $p$ of $\mathcal O_{\omega}$,  is defined and its basic properties developed in \cite[Chapter VI, Section 11]{ZS2}, \cite{K0} and \cite[Section 7.1]{CP}.
 In the case that $L$ is Galois over $K$, we have the formula 
 \begin{equation}\label{found8}
 [L:K]=e(\omega/\nu)f(\omega/\nu)\delta(\omega/\nu)g
 \end{equation}
 where $g$ is the number of  extensions of $\nu$ to $L$. In fact, 
 we have  the equation (c.f. \cite{Ku} or Section 7.1 \cite{CP})
$$
|G^s(\omega/\nu)|=e(\omega/\nu)f(\omega/\nu)\delta(\omega/\nu),
$$
where $G^s(\omega/\nu)$ is the decomposition group of $L/K$. 

If $K\rightarrow L$ is a finite Galois extension, then we will denote the Galois group of $L/K$ by $\mbox{Gal}(L/K)$.

\subsection{Initial and final segments and cuts}
We review some basic material about cuts in totally ordered sets from \cite{Ku}.
Let $(S,<)$ be a totally ordered set. An initial segment of $S$ is a subset $\Lambda$ of $S$ such that if $\alpha\in \Lambda$ and $\beta<\alpha$ then $\beta\in \Lambda$. A final segment of $S$ is a subset $\Lambda$ of $S$ such that if $\alpha\in \Lambda$ and $\beta>\alpha$ then $\beta\in \Lambda$. A cut in $S$ is a pair of sets $(\Lambda^L,\Lambda^R)$ such that $\Lambda^L$ is an initial segment of $S$ and $\Lambda^R$ is a final segment of $S$ satisfying $\Lambda^L\cup \Lambda^R=S$ and $\Lambda^L\cap \Lambda^R=\emptyset$.
If $\Lambda_1$ and $\Lambda_2$ are two cuts in $S$, write $\Lambda_1<\Lambda_2$ if $\Lambda_1^L\subsetneq \Lambda_2^{L}$.
Suppose that $S\subset T$ is an order preserving inclusion of ordered sets and $\Lambda=(\Lambda^L,\Lambda^R)$ is a cut in $S$. Then define the cut induced by $\Lambda=(\Lambda^L,\Lambda^R)$ in $T$ to be the cut 
$\Lambda\uparrow T=(\Lambda^L\uparrow T,T\setminus (\Lambda^L\uparrow T))$ where $\Lambda^L\uparrow T$ is the least initial segment of $T$ in which $\Lambda^L$ forms a cofinal subset. 

We embed $S$ in the set of all cuts of $S$ by sending $s\in S$ to 
$$
s^+=(\{t\in S\mid t\le s\},\{t\in S\mid t>s\}).
$$
we may identify $s$ with the cut $s^+$. Define
$$
s^{-}=(\{t\in S\mid t< s\},\{t\in S\mid t\ge s\}).
$$
Given a cut $\Lambda=(\Lambda^L,\Lambda^R)$, we define 
$-\Lambda=(-\Lambda^R,-\Lambda^L)$ where
$-\Lambda^L=\{-s\mid s\in \Lambda^L\}$ and $-\Lambda^R=\{-s\mid s\in \Lambda^R\}$.
We have that if $\Lambda_1$ and $\Lambda_2$ are cuts, then $\Lambda_1<\Lambda_2$ if and only if $-\Lambda_2<-\Lambda_1$.

Observe that for $s\in S$, $-s=-(s^+)=(-s)^-$ and $-(s^-)=(-s)^+=-s$.

\subsection{Distances}\label{SecDist}
Let $K\rightarrow L$ be an   extension of  fields and  $\omega$ be a valuation of $L$ with restriction $\nu$ to $K$. 
Let $\widetilde{\nu K}$ be the divisible hull of  $\nu K$. Suppose that $z\in L$. Then the distance of $z$ from $K$ is defined in \cite[Section 2.3]{Ku}  to be the cut $\mbox{dist}(z,K)$ of $\widetilde{\nu K}$ in which  the initial segment of $\mbox{dist}(z,K)$ is the least initial segment of
$\widetilde{\nu K}$ in which $\omega(z-K)$ is cofinal. That is, 
$$
\mbox{dist}(z,K)=(\Lambda^L(z,K),\Lambda^R(z,K))\uparrow\widetilde{\nu K}
$$
where 
$$
\Lambda^L(z,K)=\{\omega(z-c)\mid c\in K\mbox{ and }\omega(z-c)\in \nu K\}.
$$
The following notion of equivalence is defined in \cite[Section 2.3]{Ku}. If $y,z\in L$, then $z\sim_Ky$ if  
$\omega(z-y) >\mbox{dist}(z,K)$.

\subsection{Higher ramification groups}\label{SubsecHRG} We recall material from pages 78 and 79 of \cite{ZS2}.
Suppose that $K\rightarrow L$ is a finite Galois extension with Galois group $G$ and 
$\omega$ is a valuation of $L$. Suppose that $I\subset \mathcal O_{\omega}$ is an ideal. Associate to $I$  higher ramification subgroups of $G$ by 
$$
G_I=\{s\in G\mid s(x)-x\in I\mbox{ for every }x\in \mathcal O_{\omega}\},
$$
$$
H_I=\{s\in G\mid s(x)-x\in Ix\mbox{ for every }x\in L\}.
$$
The group $H_I$ is denoted by $G_I'$ in \cite{KP}.
We always have that $H_I\subset G_I$. If $I\subset J$ then $G_I\subset G_J$ and $H_I\subset H_J$.

\subsection{Artin-Schreier extensions}\label{Galois}
Let $K\rightarrow L$ be an  Artin-Schreier extension of  fields of characteristic $p>0$ and  $\omega$ be a valuation of $L$ with restriction $\nu$ to $K$. The field $L$ is Galois over $K$ with Galois group $G\cong\ZZ_p$, where $p$ is the characteristic of $K$.

Let $\Theta\in L$ be an Artin-Schreier generator of $K$; that is, there is an expression
$$
\Theta^p-\Theta = a
$$
for some $a\in K$. We have that 
$$
\mbox{Gal}(L/K)\cong \ZZ_p=\{{\rm id}, \sigma_1,\ldots,\sigma_{p-1}\},
$$
where $\sigma_i(\Theta)=\Theta+i$.

Since $L$ is Galois over $K$, we have that $ge(\omega/\nu)f(\omega/\nu)\delta(\omega/\nu)=p$ where $g$ is the number of extensions of $\nu$ to $L$. So  we either have that $g=1$ or $g=p$.
If $g=1$, then $\omega$ is the unique extension of $\nu$ to $L$ and either $e(\omega/\nu)=p$ and $\delta(\omega/\nu)=1$ or $e(\omega/\nu)=1$ and $\delta(\omega/\nu)=p$. In particular, the extension is defect if and only if is an immediate extension ($e=f=1$) and $\omega$ is the unique extension of $\nu$ to $L$.

For $\alpha\in \omega L$, we have an associated ideal $I_{\alpha}=\{f\in \mathcal O_{\omega}\mid \omega(f)\ge \alpha\}$.  We have that $\beta<\alpha$ implies $I_{\alpha}\subset I_{\beta}$.
Thus $\{\alpha\in \omega L\mid G_{I_{\alpha}}=1\}$ is a final segment of $\omega L$. We define the ramification cut ${\rm Ram}(\omega/\nu)$ of the valued field extension $L/K$ to be the cut in  $\widetilde{ \nu K}$ induced by this final segment; that is, ${\rm Ram}(\omega/\nu)=({\rm Ram}(\omega/\nu)^L,{\rm Ram}(\omega/\nu)^R)$ where ${\rm Ram}(\omega/\nu)^R$ is the smallest
final segment of $\widetilde{\nu K}$ in which $\{\alpha\in \omega L\mid G_{I_{\alpha}}=1\}$
is coinitial.

From now on in this subsection, suppose that $L$ is a defect extension of $K$. By \cite[Lemma 4.1]{Ku}, the distance 
$\delta=\mbox{dist}(\Theta,K)$ does not depend on the choice of Artin-Schreier generator $\Theta$, so $\delta$ can be called the distance of the Artin-Schreier extension. Since $L/K$ is an immediate extension, the set $\omega(\Theta - K)$ is an initial segment in $\nu K$ which  has no maximal element by \cite[Theorem 2.19]{Ku}.

 We have, since the extension is defect, that
\begin{equation}\label{N22}
\delta={\rm dist}(\Theta,K)\le 0^-
\end{equation}
by \cite[Corollary 2.30]{Ku}.

A defect Artin-Schreier extension $L$ is defined in \cite[Section 4]{Ku} to be a dependent defect Artin-Schreier extension if there exists an immediate purely inseparable extension $K(\eta)$ of $K$ of degree $p$ such that $\eta\sim_K\Theta$. Otherwise, $L/K$ is defined to be an independent defect Artin-Schreier defect extension. We have by \cite[Proposition 4.2]{Ku} that for a defect Artin- Schreier extension,
\begin{equation}\label{eqN23}
\mbox{$L/K$ is independent if and only if the distance $\delta={\rm dist}(\Theta,K)$  satisfies $\delta=p\delta$.}
\end{equation}

%Let $\widetilde{\nu K}$ be the divisible hull of  $\nu K$. Suppose that $z\in L$. Then the distance of $z$ from $K$ is defined in \cite{BK} to be the cut $\mbox{dist}(z,K)$ of $\widetilde{\nu K}$ in which  the initial segment of $\mbox{dist}(z,K)$ is the least initial segment of
%$\widetilde{\nu K}$ in which $\nu(z-K)$ is cofinal. 

%$L$ is Galois over $K$ with Galois group $G\cong\ZZ_p$, where $p$ is the characteristic of $K$.

\subsection{Extensions of rank 1 valuations in an Artin-Schreier extension}\label{Rank1AS}
In this subsection, we suppose that $L$ is an Artin-Schreier extension of a field $K$ of characteristic $p$, $\omega$ is a  rank 1 valuation of $L$ and $\nu$ is the restriction of $\omega$ to $K$. We suppose that $L$ is a defect extension of $K$. To simplify notation, we suppose that we have an embedding of $\nu L$ in $\RR$. %Let $\mbox{Gal}(L/K)$ be the Galois group of $L/K$.
Since $L$ has defect over $K$ and $L$ is separable over $K$, $\nu L$ is nondiscrete by the corollary on page 287 of \cite{ZS1}, so that $\nu L$ is dense in $\RR$.

%The defect $\delta(\omega/\nu)$ can be defined by the equation (c.f. \cite{Ku} or Section 7.1 \cite{CP})
%$$
%|G^s(\omega/\nu)|=[L\omega:K\nu][\omega L:\nu K]\delta(\omega/\nu),
%$$
%where $G^s(\omega/\nu)$ is the decomposition group of $L/K$. $\delta(\omega/\nu)$ is always a power of the characteristic of $K\nu$.
%The assumption that $L$ is a defect Artin Schreier extension of $K$ implies that $G^s(\omega/\nu)=\mbox{Gal}(L/K)$, so that $\omega$ is the unique extension of $\nu$ to $L$ (Extension Theorem, page 181 \cite{Neu}), and that $L$ is an immediate extension of $K$  ($L\omega=K\nu$ and $\omega L =\nu K$).

 %Let $\Theta$ be an AS generator of $K$; that is, there is an expression
%$$
%\Theta^p-\Theta = a
%$$
%for some $a\in K$. We have that 
%$$
%\mbox{Gal}(L/K)\cong \ZZ_p=\{{\rm id}, \sigma_1,\ldots,\sigma_{p-1}\},
%$$
%where $\sigma_i(\Theta)=\Theta+i$.
%Since $L/K$ is an immediate extension, the set $\omega(\Theta - K)$ is an initial segment in $\nu K$ which  has no maximal element by Lemma 1.1 \cite{Ku}. 

We define a cut 
%$d(\omega/\nu)$ 
in $\RR$ by extending the cut $\mbox{dist}(\Theta,K)$ in $\widetilde{\nu K}$ to a cut of $\RR$ by taking the initial segment of 
%$d(\omega/\nu)$
 the extended cut to be the least initial segment of $\RR$ in which  the cut $\mbox{dist}(\Theta,K)$ is confinal.
 This cut is then $\mbox{dist}(\Theta,K)\uparrow \RR$.  This cut is either $s$ or $s^{-}$ for some $s\in \RR$. If $L$ is a defect extension of $K$ then $\mbox{dist}(\Theta,K)\uparrow \RR=s^{-}$ where
$s$ is a non positive real number by \cite[Theorem 2.19]{Ku} and \cite[Corollary 2.30]{Ku}. 
%since $L/K$ is a defect extension. 
We will set 
$\mbox{dist}(\omega/\nu)$ to be this real number  $s$, so that 
$$
\mbox{dist}(\Theta,K)\uparrow \RR=s^{-}=(\mbox{dist}(\omega/\nu))^{-}.
$$
The real number $\mbox{dist}(\omega/\nu)$  is well defined since it is independent of choice of Artin-Schreier generator of $L/K$ by Lemma 4.1 \cite{Ku}.

With the assumptions of this subsection, by (\ref{N22}) and (\ref{eqN23}), the distance $\delta=\mbox{dist}(\Theta,K)$ of an Artin-Schreier extension 
is $\le$ $0^{-}$ when the extension has defect. If it is a defect extension with distance equal to $0^{-}$ then it is an  independent defect extension. If it is a defect extensions and the distance is less than $0^{-}$ then the extension is a dependent defect extension. Thus if $L/K$ is a defect extension, we have that $\mbox{dist}(\omega/\nu)\le 0$ and the defect extension $L/K$ is independent if and only if $\mbox{dist}(\omega/\nu)= 0$.

\section{Extensions of two dimensional regular local rings}\label{Sec2Ext}

Suppose that $M$ is a two dimensional algebraic function field over an algebraically closed field $k$ of characteristic $p>0$ and $\mu$ is a nondiscrete rational rank 1 valuation of $M$.
Suppose that $A$ is an algebraic regular local ring of $M$ such that $\mu$ dominates $A$. A quadratic transform of $A$ is an extension $A\rightarrow A_1$ where $A_1$ is a local ring of the blowup of the maximal ideal of $A$ such that $A_1$ dominates $A$ and $A_1$ has dimension two. A quadratic transform $A\rightarrow A_1$ is said to be along the valuation $\mu$ if $\mu$ dominates $A_1$.

Let
$$
A=A_0\rightarrow A_1\rightarrow A_2\rightarrow \cdots
$$
be the sequence of quadratic transforms along $\mu$ so that the valuation ring $\mathcal O_{\mu}=\cup A_i$ (by \cite[Lemma 12]{Ab1}).  Let $P$ be a height one prime ideal in $A$   such that $A/P$ is a regular local ring. $A_i$ is said to be free if  the radical of $PA_i$ is a height one prime ideal (so that $A_i/\sqrt{PA_i}$ is a regular local ring). In particular, $A_0$ is free.

Now there exist (as explained in more detail in   \cite[Definition 7.11]{CP}) positive integers  $r_i'$ and $\overline r_i$ such that $A_0=A_{r_1'}$ and
 for all $i\ge 1$,
$r_i'\le \overline r_i<r_{i+1}'-1$ 
 such that   if $r_i'\le j\le \overline r_i$ then $A_j$ is free and if $\overline r_i<j<r_{i+1}'$ then $A_j$ is not free.

Suppose that $A_j$ is free. Then there exists $i$ such that $r_i'\le j\le\overline r_i$. Let $u,v$ be regular parameters in $A_j$ such that $u=0$ is a local equation of $Z(PA_j)$; that is, $(u)=\sqrt{PA_j}$. We will say that $u,v$ are allowable parameters in $A_i$.
Let $\overline u,\overline v$ be defined by 
$$
u=\overline u^{m}(\overline v+\alpha)^{a'}, v=\overline u^q(\overline v+\alpha)^{b}
$$
where 
$$
\frac{\nu(u)}{\nu(v)}=\frac{m}{q}
$$
with $\mbox{gcd}(m,q)=1$ and $0\ne\alpha\in k$ is such that $\nu(v)>0$. 

Then there exists $k>j$ such that $\overline u$ and $\overline v$ are regular parameters in $A_k$. Further, $A_k$ is free, and $\overline u=0$ is a local equation of the reduced exceptional divisor of $\mbox{Spec}(A_k)\rightarrow \mbox{Spec}(A)$, so that $\overline u, \overline v$ are admissible parameters in $A_k$. 
If $\mu(v)\not\in \mu(u)\ZZ$, then $m>1$ and $k=r_{i+1}'$. If $\mu(v)\in \mu(u)\ZZ$ then $m=1$ and $k\le \overline r_i$.

A regular system of parameters with $\mu(v)\not\in \mu(u)\ZZ$ can always be found from a given regular system of parameters $u,v$ by possibly replacing $v$ with the difference of $v$ and a suitable polynomial $g(u)\in k[u]$ (which necessarily has no constant term).

We will call the sequence $A_0=A_{r_0'}\rightarrow A_{r_1'}\rightarrow A_{r_2'}\rightarrow\cdots$ the sequence of standard sequences of quadratic transforms along $\mu$. Observe that this sequence depends on the  choice of $P$ in $A$.

Let $K\rightarrow L$ be a finite separable extension of two dimensional algebraic function fields over an algebraically closed field $k$. Suppose that $R$ is a two dimensional regular algebraic local ring of $K$ and $S$ is a two dimensional regular algebraic local ring of $K^*$ such that $S$ dominates $R$. Let $P$ be a height one prime ideal in $R$ such that $R/P$ is a regular local ring and let $Q$ be a height one prime ideal in $S$ such that $S/Q$ is a regular local ring. We do not insist in this definition that the good condition that $Q\cap R=P$ holds. 

Let
$R\rightarrow R_1\rightarrow \cdots\rightarrow R_r$
and 
 $S\rightarrow S_1\rightarrow \cdots \rightarrow S_s$
  be sequences of quadratic transforms such that $S_s$ dominates $R_r$. Let $E_i$ be the  divisor $Z(PR_i)$ and $F_j$ be the  divisor  $Z(QS_j)$.

\begin{Definition} \label{found1}(\cite[Definition 7.5]{CP}) Suppose that $S_s$ dominates $R_r$.
The map $R_r\rightarrow S_s$ is said to be prepared if both $R_r$ and $S_s$ are free, the critical locus of $\mbox{Spec}(S_s)\rightarrow \mbox{Spec}(R_r)$  is contained in  $F_s$  and we have an expression
$u=\gamma x^a$, where $u$ is part of a regular system of parameters of $R_r$ such that $u=0$ is a local equation of $E_r$, $x$ is part of a regular system of parameters of $S_s$ such that $x=0$ is a local equation of $F_s$ and $\gamma$ is a unit in $S_s$.
\end{Definition}

%Maybe we don't need the following:
%\begin{Definition}\label{found2}(\cite[Definition 7.8]{CP})
%Suppose that $\nu$ is a rational rank 1 valuation of $K^*$ such that $\nu$ dominates $S_s$ and $R_r\rightarrow S_s$ is  prepared. A regular system of parameters $(u,v)$ of $R_r$ is said to be admissible if the support of $u$ is equal to $E_r$ and and if $\nu(v)$ is maximal for all such regular system of parameters containing $u$.
%\end{Definition}

 Suppose that $R_r\rightarrow S_s$ is prepared. Let $(u,v)$ and $(x,y)$ be admissible parameters in $R_r$ and $S_s$ respectively; that is,  $u$ is part of a regular system of parameters of $R_r$ such that $u=0$ is a local equation of $E_r$, $x$ is part of a regular system of parameters of $S_s$ such that $x=0$ is a local equation of $F_s$. Further,   $u=\gamma x^a$, where $\gamma$ is a unit in $S_s$.
Then there is an expression 
\begin{equation}\label{found6}
u=\gamma x^a,
v=x^b f
\end{equation}
 where $\gamma\in S_s$ is a unit, $f\in S_s$ and $x$ does not divide $f$ in $S_s$. 
 
 \begin{Definition}\label{DefN2}  We will say that $R_r\rightarrow S_s$ is well prepared if  $f$ is not a unit in $S_s$.
 \end{Definition}

 Suppose that $R_r\rightarrow S_s$ is well prepared. 
The complexity of $R_r\rightarrow S_s$ is $ad$ where $d$ is the order of the residue of $f$ in the one dimensional regular local ring $S_s/(x)$.

 The complexity is defined in \cite[Definition 7.9]{CP}. It is shown there that the complexity depends only on the extension $R_r\rightarrow S_s$. 
 
% \begin{Theorem}\label{found4}(\cite[Section 7.7]{CP}) Suppose that $a\ge r$, $b\ge s$, $S_s$ dominates $R_r$, $S_b$ dominates $R_a$ and $R_r\rightarrow S_s$ and $R_a\rightarrow S_b$ are well prepared. Then the complexity of $R_a\rightarrow S_b$ is less than or equal to the complexity of $R_r\rightarrow S_s$.
 %\end{Theorem}

 \begin{Proposition}\label{found7}(\cite[Proposition 7.2]{CP}) 
 %Suppose that $\omega$ is a rational rank 1 nondiscrete valuation of the quotient field of $K^*$ with restriction $\nu$ to $K$ such that $\omega$ dominates $S_j$ for all $j$ and $\nu$ dominates $R_i$ for all $i$.
 Suppose that  $S_s$ dominates $R_r$, $R_r\rightarrow S_s$ is well prepared  and $R_r$ and $S_s$ have admissible parameters $(u,v)$ and $(x,y)$ satisfying the equation (\ref{found6}).   Let $S^*$ be the unique two dimensional algebraic normal local ring over $k$ lying above $R$ with $\mbox{QF}(S^*)=K^*$ and $S^*\subset S$. Then there exists a commutative diagram 
 $$
 \begin{array}{ccc}
 S^*&\rightarrow &S_s\\
 \uparrow&&\uparrow\\
 R_r&\rightarrow &R^*
 \end{array}
 $$
 where all arrow are dominating maps such that $R^*$ is a two dimensional algebraic normal local ring of $K$ such that $S_s$ lies over $R^*$. We have
 $$
 [QF(\hat S_s):QF(\hat R^*)]=ad
 $$
 where  $d$ is the order of the residue of $f$ in the one dimensional regular local ring $S_s/(x)$.
 \end{Proposition}
 
 \begin{Remark}\label{found9} Suppose that assumptions are as in Proposition \ref{found7} with the additional assumption that $K\rightarrow L$ is Galois. Then the complexity $ad$ of $R_r\rightarrow S_s$ divides the degree $[L:K]$.
 \end{Remark}
 
 \begin{proof} Let $\overline S$ be the integral closure of $R^*$ in $L$. Let $m_1,\ldots,m_g$ be the maximal ideals of $\overline S$. The local ring $S_s$ is one of the localizations $\overline S_{m_i}$. The Galois group $G(L/K)$ acts transitively on the local rings $\overline S_{m_i}$ so these rings are all isomorphic $k$-algebras. The $m_{R^*}$-adic completion of $\overline S$ is the direct sum of the complete local rings 
 $\widehat{\overline S_{m_i}}$, and one of these local rings is $\hat S_s$.  We have that 
 $$
 [L:K]=\sum_{i=1}^g [\mbox{QF}(\widehat{\overline S_{m_i}}):\mbox{QF}(\hat R^*)]
 $$
 by \cite[Proposition 1]{Ab2}. Thus $[L:K]=g[QF(\hat S_s):QF(\hat R^*)]=gad$.
 \end{proof}

 The following remark follows from \cite[Theorem 2]{Ab1}.
 
 \begin{Remark}\label{RemarkN1}
 Suppose that $\omega$ is a rational rank 1 nondiscrete valuation of the quotient field of $K^*$ with restriction $\nu$ to $K$ such that $\omega$ dominates $S_j$ for all $j$ and $\nu$ dominates $R_i$ for all $i$.
 Then given $r_0>0$ and $s_0>0$ there exist $r\ge r_0$ and $s\ge s_0$ such that  $S_s$ dominates $R_r$ and  $R_r\rightarrow S_s$ is well prepared.  This result is true for any initial choice of $P$ in $R$ and $Q$ in $S$.
 \end{Remark}  
 
\begin{Proposition}\label{Prop2*} Let $K\rightarrow L$ be a finite separable extension of two dimensional algebraic function fields over an algebraically closed field $k$ of characteristic $p>0$, $\omega$ is a rational rank 1 nondiscrete valuation of $L$ (with residue field $k$) and $\nu$ is the restriction of $\omega$ to $K$.

 Suppose that $A$ is an algebraic  local ring of $K$ which is dominated by $\nu$. Then there exists an algebraic regular local ring $R'$ of $K$ which is dominated by $\nu$ and dominates $A$ such that if $R$ is a regular algebraic local ring of $K$ which dominates $R'$ and $S$ is a regular algebraic local ring of $L$ which is dominated by $\omega$ and dominates $R$ such that there are regular parameters $x,y$ in $S$ and $u,v$ in $R$ such that $u=\gamma x^a$ and $v=x^bf$ with $\gamma$ a unit in $S$ and  $f\in S$ such that $f$ is not a unit in $S$,   $x$ does not divide $f$ and the critical locus of $\mbox{Spec}(S)\rightarrow\mbox{Spec}(R)$ is contained in $Z(xS)$.        Then 
 letting $d=\mbox{dim}_k S/(f,x)$, we have that the complexity $ad=e(\omega/\nu)\delta(\omega/\nu)$.
 
\end{Proposition}

This is proved in \cite[Section 7.9]{CP}  and \cite[Proposition 3.4]{C1}.

Let $K\rightarrow L$ be a finite separable extension of two dimensional algebraic function fields over an algebraically closed field $k$ of characteristic $p>0$, $\omega$ is a rational rank 1 nondiscrete valuation of $L$ (with residue field $k$) and $\nu$ is the restriction of $\omega$ to $K$. Let $R$ be a regular algebraic local ring of $K$  and $S$ be a regular algebraic local ring of $L$ which is dominated by $\omega$ and dominates $R$.  Let $P$ be a height one prime ideal in $R$ such that $R/P$ is a regular local ring and let $Q$ be a height one prime ideal in $S$ such that $S/Q$ is a regular local ring.
Let
\begin{equation}\label{found10}
R\rightarrow R_1\rightarrow \cdots\rightarrow \cdots
 \mbox{ and }
 S\rightarrow S_1\rightarrow \cdots \rightarrow \cdots
 \end{equation}
 be the infinite sequences of quadratic transforms along $\nu$ and $\omega$ respectively.
 By Remark \ref{RemarkN1} and Proposition \ref{Prop2*}, there exits a positive integer $r_0$ such that whenever $r\ge r_0$ and $R_r\rightarrow S_s$ is well prepared, we have that the complexity $ad$ of this extension is equal to $e(\omega/\nu)\delta(\omega/\nu)$. We will call this the stable complexity of the sequences (\ref{found10}).

Suppose that $K\rightarrow L$ is a finite extension of two dimensional algebraic function fields,
 $R$ is an algebraic  regular local ring of $K$ which is dominated by a regular algebraic local ring $S$ of $L$ such that $\dim R=\dim S=2$. Let $x,y$ be regular parameters in $S$ and $u,v$ be regular parameters in $R$. Then we can form the Jacobian ideal 
$$
J(S/R)=
(\frac{\partial u}{\partial x}\frac{\partial v}{\partial y}-\frac{\partial u}{\partial y}\frac{\partial v}{\partial x}).
$$
This ideal is independent of choice of regular parameters.

The following proposition is established in \cite{CP}.

\begin{Proposition}\label{Prop10*} 
Suppose that $A$ is an algebraic  local ring of $K$ and $B$ is an algebraic local ring of $L$ which is dominated by a rational rank 1 nondiscrete valuation $\omega$ of $L$
such that $B$ dominates $A$. Then there exists a commutative diagram of homomorphisms
$$
\begin{array}{ccc}
R&\rightarrow &S\\
\uparrow&&\uparrow \\
A&\rightarrow &B
\end{array}
$$
such that $R$ is a regular algebraic local ring of $K$ with regular parameters $u,v$, $S$ is a regular algebraic local ring of $L$ with regular parameters $x,y$ such that $S$ is dominated by $\omega$, $S$ dominates $R$, $J(S/R)= (x^{\overline c})$ for some non negative integer $\overline c$ and there is an expression
$$
u=\gamma x^a, v=x^b(y^n\tau+x\Omega)
$$
where $\tau,\gamma$ are units in $S$, $\Omega\in S$ and $n>0$, $0\le b<a$.
Thus the quadratic transform of $R$ along $\nu$ is not dominated by $S$.
\end{Proposition}

\begin{proof} By two dimensional local uniformization (or resolution of singularities), \cite{Ab2}, \cite{Li} or \cite{CJS}, there exists a commutative diagram
$$
\begin{array}{ccc}
R_0&\rightarrow &S_0\\
\uparrow&&\uparrow \\
A&\rightarrow &B
\end{array}
$$
such that $S_0$ is an algebraic local ring of $L$ which dominates $B$, $R_0$ is an algebraic regular local ring of $K$ which dominates $A$, $\omega$ dominates $S_0$ and $S_0$ dominates $R_0$.  Since $\nu$ has rational rank 1 and is nondiscrete, there exists a nontrivial sequence of quadratic transforms $R_0\rightarrow R_1$ along $\nu$ such that $R_1$ is free. Let $E$ be the last exceptional divisor of the sequence of quadratic transforms factoring $R_0\rightarrow R_1$ and let $u,v$ be regular parameters in $R_1$ such that $u=0$ is a local equation of $E$ in $\mbox{Spec}(R_1)$. We have that $J(R_1/R_0)=(u^g)$ for some positive integer $g$. Let $S_0\rightarrow S$ be a sequence of quadratic transforms along $\omega$ such that $S$ is free and $S$ dominates $R_1$, and the support of $J(S_0/R_0)S$ is the last exceptional divisor $F$ of the sequence of quadratic transforms factoring     $S_0\rightarrow S$. Let $x,y$ be regular parameters in $S$ such that $x=0$ is a local equation of $F$ in $\mbox{Spec}(S)$. Thus $J(S/R_0)=J(S_0/R_0)J(S/S_0)=(x^f)$ for some positive integer $f$. Now 
$(x^f)=J(S/R_0)=J( R_1/R_0)J(S/R_1)=u^gJ(S/R_1)$ implies $u=\gamma x^l$ for some unit $\gamma$ in $S$ and positive integer $l$. Thus there exist $b,n\in \NN$, a unit $\tau$ in $S$ and $\Omega\in S$ such that $v=x^b(\tau y^n+x\Omega)$. There exists a sequence of quadratic transforms $R_1\rightarrow R_2$ along $\nu$ such that $S$ dominates $R_2$, and after replacing $R_1$ with $R_2$ we have that  $R_1\rightarrow S$ is well prepared ($n>0$). If $\lfloor \frac{b}{l}\rfloor=0$ then the quadratic transform of $R_1$ along $\nu$ is not dominated by $S$ and we set $R=R_1$. If $\lfloor \frac{b}{l}\rfloor>0$, then set $e=\lfloor \frac{b}{l}\rfloor$ , and let $R_1\rightarrow R$ be the sequence of $e$ quadratic transforms along $\nu$. Then $R\rightarrow S$ satisfies the conclusions of the proposition.
\end{proof}

\begin{Remark}\label{RemarkAS} 
Let $K\rightarrow L$ be an Artin-Schreier    extension of two dimensional algebraic function fields over an algebraically closed field $k$ of characteristic $p>0$.  Let  $\omega$ be a rational rank 1 nondiscrete valuation of $L$ (with residue field $k$) and $\nu$ be the restriction of $\omega$ to $K$.
Since $L$ is Galois over $K$, we have that $g(\omega/\nu)e(\omega/\nu)\delta(\omega/\nu)=p$ where $g=g(\omega/\nu)$ is the number of extensions of $\nu$ to $L$. So  we either have that $g=1$ or $g=p$.
If $g=1$, then $\omega$ is the unique extension of $\nu$ to $L$ and either $e(\omega/\nu)=p$ and $\delta(\omega/\nu)=1$ or $e(\omega/\nu)=1$ and $\delta(\omega/\nu)=p$. If $g=1$, we have by Proposition \ref{Prop2*} that the stable complexity of the sequences (\ref{found10}) is $ad=p$.
If $g=p$, then $e(\omega/\nu)=1$ and $\delta(\omega/\nu)=1$ and the stable complexity of the sequences (\ref{found10})is $ad=1$. 
\end{Remark}

The following proposition is proven in  \cite{Pi}.

\begin{Proposition}\label{Prop1*} 
Suppose that $K\rightarrow L$ is an Artin-Schreier extension of two dimensional algebraic function fields over an algebraically closed field $k$ of characteristic $p>0$, $\omega$ is a rational rank 1 nondiscrete valuation of $L$ with restriction $\nu=\omega|K$. Further suppose that $A$ is an algebraic  local ring of $K$ and $B$ is an algebraic local ring of $L$ which is dominated by $\omega$ such that $B$ dominates $A$. Then there exists a commutative diagram of homomorphisms
$$
\begin{array}{ccc}
R&\rightarrow &S\\
\uparrow&&\uparrow \\
A&\rightarrow &B
\end{array}
$$
such that $R$ is a regular algebraic local ring of $K$ with regular parameters $u,v$, $S$ is a regular algebraic local ring of $L$ with regular parameters $x,y$ such that $S$ is dominated by $\omega$, $S$ dominates $R$, $R\rightarrow S$ is well prepared with admissible parameters $u,v$ in $R$ and $x,y$ in $S$ (with respect to the prime ideals $P=uR$ in $R$ and $Q=xS$ in $S$). We further have  that $R\rightarrow S$ is quasi finite,
$J(S/R)= (x^{\overline c})$ for some non negative integer $\overline c$ and one of the following three cases holds:c
\begin{enumerate}
\item[0)]  $u=x$, $v=y$ ($R\rightarrow S$ is unramified).
\item[1)]  $u=x$, $v=y^p\gamma+x\Sigma$ where $\gamma$ is a unit in $S$ and $\Sigma\in S$.
\item[2)] $u=\gamma x^p$, $v=y$ where $\gamma$ is a  unit in $S$.
% and $\Omega\in S$.
\end{enumerate}

\end{Proposition}

\begin{proof} By Proposition \ref{Prop10*} and Remark \ref{found9}, we  may construct a diagram
$$
\begin{array}{ccc}
\overline R&\rightarrow & \overline S\\
\uparrow&&\uparrow\\
A&\rightarrow&B
\end{array}
$$
such that all the conclusions of the proposition hold, except possibly $\overline R\rightarrow \overline S$ is not quasi finite, and we have a form 
\begin{equation}\label{eq22*}
u=\delta x^p, v = x^by
\end{equation}
where $b$ is an integer with $0<b<p$ and $\delta$ is a unit in $\overline S$. We will show that after one more sequence of blowups, we obtain a map of the form of 0), 1) or 2).

There exists a sequence of quadratic transforms  $S\rightarrow S'$ of regular local rings along $\omega$ such that $S'$ has regular parameters $\tilde x_1,\tilde y_1$ defined by 
\begin{equation}\label{eq100}
x=\tilde x_1^{\overline a}(\tilde y_1+\alpha)^{\overline a'},
y=\tilde x_1^{\overline b}(\tilde y_1+\alpha)^{\overline b'}
\end{equation}
where $0\ne \alpha\in k$ and $\overline a\overline b'-\overline a'\overline b=1$. Then we have an expression
$$
u=\delta \tilde x_1^{d_1}(\tilde y_1+\alpha)^{e_1},
v=\tilde x_1^{f_1}(\tilde y_1+\alpha)^{g_1}
$$
where $d_1g_1-e_1f_1=p$. Let $l=\mbox{gcd}(d_1,f_1)$. The number $l$  must either be 1 or $p$. We have that
$$
\delta=\delta_0+\tilde x_1 \Omega\mbox{ for some $0\ne\delta_0\in k$ and $\Omega\in S'$.}
$$
We have a  sequence of quadratic transforms  $R\rightarrow R'$ of regular local rings such that 
$S'$ dominates $R'$ and $R'$ has regular parameters $\tilde u$ and $\tilde v$ such that
$$
\tilde u=\tilde x_1^l(\tilde y_1+\alpha)^{e_2}\delta^{f_2},
\tilde v=(\tilde y_1+\alpha)^{g_2}\delta^{h_2}-\alpha^{g_2}\delta_0^{h_2}
$$
We have that $lg_2=p$ as
$$
p=\left|
\begin{array}{cc} 
d_1&e_1\\
f_1&g_1
\end{array}\right|
=\left|\begin{array}{cc}
l&e_2\\
0&g_2
\end{array}\right|.
$$
First suppose that $l=p$. Then $g_2=1$ and we have an expression of the form of 2). 
Now suppose that $l\ne p$. Then $l=1$ and  we have an expression of the form of  1).
\end{proof}

\section{Some calculations in two dimensional Artin-Schreier extensions}\label{SecCalc}

Let $K\rightarrow L$ be an Artin-Schreier extension of two dimensional algebraic function fields over an algebraically closed field $k$ of characteristic $p>0$. Let $R\rightarrow S$ be an extension from a regular algebraic local ring of $K$ to a regular algebraic local ring of $L$ such that $S$ dominates $R$. 

Let $u, v$ be regular parameters in $R$ and $x,y$ be regular parameters in $S$. 
We will say that $R\rightarrow S$ is of type 0 with respect to these parameters if 
$$
\mbox{Type 0:}\,\,\,\,\, u=\gamma x, v=y\tau +x\Omega
$$
where $\gamma,\tau$ are units in $S$ and $\Omega\in S$, so that $R\rightarrow S$ is unramified.
We will say that $R\rightarrow S$ is of type 1 with respect to these parameters if 
$$
\mbox{Type 1:}\,\,\,\,\, u=\gamma x, v=y^p\tau +x\Omega
$$
where $\gamma,\tau$ are units in $S$ and $\Omega\in S$.
We will say that $R\rightarrow S$ is of type 2 with respect to these parameters if 
$$
\mbox{Type 2:}\,\,\,\,\, u=\gamma x^p, v=y\tau +x\Omega
$$
where $\gamma,\tau$ are units in $S$ and $\Omega\in S$.

These definitions are such that if one these types hold, and $\overline u,\overline v$ are regular parameters in $R$, $\overline x,\overline y$ are regular parameters in $S$  
such that $\overline u$ is a unit in $R$ times $u$ and $\overline x$ is a unit in $S$ times $x$ then $R\rightarrow S$ is of the same type for the new parameters $\overline u,\overline v$ and $\overline x,\overline y$.

%By Proposition \ref{Prop1*} after enough blowing up along $\omega$ and $\nu$ we can always obtain one of the above three types.
If  we replace $R\rightarrow S$ with a well prepared map $R_i\rightarrow S_j$ in the sequences (\ref{found10}), we will insist that the above parameters be admissible. We see that the three types are preserved by allowable changes of variables. In particular, we may obtain the respective forms of Proposition \ref{Prop1*} by an allowable change of variables.

%In keeping with the notation of Proposition \ref{Prop1*}, we will say that $R\rightarrow S$ is of type 0 if there exist regular parameters $(u,v)$ in $R$ and $(x,y)$ in $S$ such that $u=x$ and $v=y$. We will say that $R\rightarrow S$ is of type 1 if there exist regular parameters $(u,v)$ in $R$ and $(x,y)$ in $S$ such that
%$u=x$ and $v=y^p\gamma+x\Sigma$ where $\gamma$ is a unit in $S$ and $\Sigma\in S$. We will say that $R\rightarrow S$ is of type 2 if there exist regular parameters $(u,v)$ in $R$ and $(x,y)$ in $S$ such that $u=\gamma x^p$ and $v=y\tau+x\Omega$ where $\gamma,\tau$ are units in $S$ and $\Omega\in S$.

\begin{Theorem}\label{TheoremA} Suppose that $R\rightarrow S$  is of type 1 with respect to regular parameters $x, y$ in $S$ and $u,v$ in $R$ and that $J(S/R)=(x^{\overline c})$. Let $\overline x=u$, $\overline y=y-g(\overline x)$ 
where $g(\overline x)\in k[\overline x]$ is a polynomial with zero constant term, 
so that $\overline x,\overline y$ are regular parameters in $S$. Computing the Jacobian determinate $J(S/R)$, we see that 
$$
u=\overline x, v=\overline y^p\gamma+\overline x^{\overline c}\overline y\tau+f(\overline x)
$$
 where $\gamma,\tau$ are unit series in $\hat S$ and $f(\overline x)=\sum e_i\overline x^i\in k[[\overline x]]$.
Make the change of variables $\overline v=v-\sum e_iu^i$ where the sum is over $i$ such that $i\le \frac{pq}{m}$ so that $u,\overline v$ are regular parameters in $R$.
%Then $R\rightarrow S$ is of type 1 with respect to the regular parameters $x,\overline y$ and $u,v$.

Suppose that $m,q$ are positive integers with $m>1$ and $\mbox{gcd}(m,q)=1$. Let $\alpha$ be a nonzero element of $k$.
Consider the sequence of quadratic transforms $S\rightarrow S_1$ so that $S_1$ has regular parameters $x_1,y_1$ defined by
$$
\overline x=x_1^{m}(y_1+\alpha)^{a'}, \overline y=x_1^{q}(y_1+\alpha)^{b'}
$$
where  $a',b'\in \NN$ are such that $mb'-q a'=1$. 

 We have that $R\rightarrow S$ is of type 1 with respect to the regular parameters $\overline x,\overline y$ and $u,v$.
Let $\sigma=\mbox{gcd}(m,p q)$ which is 1 or $p$. 

There exists a unique   sequence of quadratic transforms $R\rightarrow R_1$ such that $R_1$ has regular parameters $u_1,v_1$ defined by 
$$
u=u_1^{\overline m}(v_1+\beta)^{c'}, \overline v=u_1^{\overline q}(v_1+\beta)^{d'}
$$
with $0\ne\beta\in k$
giving
  a commutative diagram of homomorphisms
$$
\begin{array}{lll}
R_1&\rightarrow &S_1\\
\uparrow&&\uparrow\\
R&\rightarrow &S
\end{array}
$$
such that $R_1\rightarrow S_1$ is quasi finite. 
We  have that  $J(S_1/R_1)=(x_1^{c_1})$ for some positive integer $c_1$ and $R_1\rightarrow S_1$ is quasi finite. 
Further:

\begin{enumerate}
 \item[0)] If $\frac{q}{m}\ge  \frac{\overline c}{p-1}$  then $R_1\rightarrow S_1$ is of type 0.
\item[1)] If $\frac{q}{m}< \frac{\overline c}{p-1}$ and $\sigma=1$ then $R_1\rightarrow S_1$ is of type 1 and
$$
\left(\frac{c_{1}}{p-1}\right)=\left(\frac{\overline c}{p-1}\right)m-q.
$$  
\item[2)] If $\frac{q}{m}< \frac{\overline c}{p-1}$ and $\sigma=p$ then $R_1\rightarrow S_1$ is of type 2 and 
$$
\left(\frac{c_{1}}{p-1}\right)=\left(\frac{\overline c}{p-1}\right)m-q+1.
$$
\end{enumerate}
In cases 1) and 2), $m=\sigma \overline m$, $pq=\sigma\overline q$ and  $\overline mc'-\overline q d'=1$. 
\end{Theorem}

\begin{proof}
%There exists a unique polynomial of smallest degree $\sum a_ix^i$ in $k[x]$ such that setting $\overline y=y-\sum a_ix^i$, we have that $\omega(\overline y)\not\in \omega(x)\ZZ$. Computing the Jacobian determinate $J(S/R)$, we see that 
%$$
%u=x, v=\overline y^p\gamma+x^{\overline c}\overline y\tau+f(x)
%$$
 %where $\gamma,\tau$ are unit series in $\hat S$ and $f(x)=\sum e_ix^i\in k[[x]]$.
%Make the change of variables $\overline v=v-\sum e_iu^i$ where the sum is over $i$ such that $i\le \frac{pq}{m}$.

%%%%%%%%%%%%%%%%%
Define a monomial valuation $\mu$ dominating $\hat S$ by prescribing that $\mu(\overline x)=m$, $\mu(\overline y)=q$ and for $0\ne \sum a_{ij}\overline x^i\overline y^j\in \hat S$, $\nu(\sum a_{ij}\overline x^i\overline y^j)=\min\{im+jq\mid a_{ij}\ne 0\}$.

Expand 
\begin{equation}\label{eq6}
\overline v=\sum_{i=1}^d \gamma_i\overline x^{\alpha_i}\overline y^{\beta_i}+\sum_{i>d}\gamma_i\overline x^{\alpha_i}\overline y^{\beta_j}
\end{equation}
where all $\gamma_i\in k$ are nonzero, $\overline x^{\alpha_i}y^{\beta_i}$ have minimal $\mu$ value $\rho$ for $1\le i\le d$ and $\overline x^{\alpha_i}\overline y^{\beta_i}$ have value larger than $\rho$ for $i>d$ and $\beta_1<\cdots<\beta_d$. By our choice of $\overline v$  we have  that 
\begin{equation}\label{eq7}
\beta_1>0.
\end{equation}
Substitute
 \begin{equation}\label{eq5}
\overline x=x_1^{m}(y_1+\alpha)^{a'}, \overline y=x_1^{q}(y_1+\alpha)^{b'}.
\end{equation}
into $u$ and the expression (\ref{eq6}) of $\overline v$ to obtain
$$
\begin{array}{lll}
u&=& x_1^{m}(y_1+\alpha)^{a'}\\
\overline v&=&x_1^{\alpha_1m+\beta_1q}(y_1+\alpha)^{a'\alpha_1+b'\beta_1}\Lambda
\end{array}
$$
where 
\begin{equation}\label{eq20}
\Lambda=\left(\sum_{i=1}^d\gamma_i(y_1+\alpha)^{\frac{\beta_i-\beta_1}{m}}+x_1\Omega\right)
\end{equation}
with  $\frac{\beta_i-\beta_1}{m}\in\NN$ for all $i$. This expression for $\overline v$ is shown in  \cite{Z} and in the proof of \cite[Theorem 8.4]{RS}.

Assume that $\Lambda$ is a unit. We will show later in this proof  that with our choice of variables, if $\Lambda$ is not a unit then we will reach the case where $R_1\rightarrow S_1$ is of type 0  in the conclusions of the theorem.

Define
\begin{equation}\label{eq11}
\overline \sigma=\mbox{gcd}(m,\alpha_1m+\beta_1q)=\mbox{gcd}(m,\beta_1q).
\end{equation}
Let 
\begin{equation}\label{eq1}
\tau=\mbox{Det}\left(\begin{array}{cc}
m&a'\\
\alpha_1 m+\beta_1q& \alpha_1a'+\beta_1b'\end{array}\right)=
\beta_1(mb'-a'q)=\beta_1>0.
\end{equation}
Let
\begin{equation}\label{eq8}
\phi=\mbox{Det}\left(\begin{array}{cc}
m&0\\
\alpha_1 m+\beta_1q& 1\end{array}\right)
=m>0.
\end{equation}

Let $R\rightarrow R^*$ be  defined by $u=u_1^{g}v_1^{g'}$, $\overline v=u_1^{h}\overline v_1^{h'}$ where $g,g',h,h'\in \NN$, $gh'-hg'=\pm 1$ and 
$$
u_1=x_1^{\overline \sigma}(y_1+\alpha)^c\Lambda^e,
\overline v_1=x_1^{\overline \sigma}(y_1+\alpha)^d\Lambda^f
$$
where 
\begin{equation}\label{eq3*}
\mbox{Det}\left(\begin{array}{cc}
\overline\sigma &c\\
\overline\sigma&d
\end{array}\right)
=\overline\sigma(d-c)=\tau=\beta_1
\end{equation}
and 
\begin{equation}\label{eq9}
\mbox{Det}\left(\begin{array}{cc}
\overline\sigma &e\\
\overline\sigma&f
\end{array}\right)
=\overline\sigma(f-e)=\phi=m.
\end{equation}

Now perform a single quadratic transform $R^*\rightarrow R_1$ so that $R_1$ is dominated by $S_1$ and  $R_1$ has regular parameters $u_1,v_1$ satisfying 
\begin{equation}\label{eq10}
u_1=x_1^{\overline\sigma}(y_1+\alpha)^{c}\Lambda^e,
v_1=\frac{\overline v_1}{u_1}-\Lambda(0,0)^{f-e}\alpha^{d-c}=(y_1+\alpha)^{d-c}\Lambda^{f-e}-\Lambda(0,0)^{f-e}\alpha^{d-c}.
\end{equation}
We have an expression for $R\rightarrow R_1$ of the form
$$
u=u_1^{\overline m}(v_1+\beta)^{\overline a_1'}, \overline v=u_1^{\overline q}(v_1+\beta)^{\overline b_1'}
$$
where $\overline m=\frac{m}{\overline\sigma}$, $\overline q=\frac{\alpha_1m+\beta_1q}{\overline\sigma}$
and $\beta=\gamma_1^{\frac{m}{\overline\sigma}}\alpha^{\frac{\beta_1}{\overline\sigma}}$.
We have that $\left(\frac{\overline v_1}{u_1}\right)(0,y_1)$ is a polynomial in $y_1$ since $d-c>0$ and $f-e>0$.

%$$
%\deg_{y_1}\left(\frac{\overline v_1}{u_1}\right)(0,y_1)=\deg_{y_1}(y_1+\alpha)^{d-c}\Lambda(0,y_1)^{f-e}\ge d-c>0.
%$$
%Hence
%$$
%0<d_1=\mbox{ord}_{y_1}v_1(0,y_1)<\infty.
%$$

%%%%%%%%%%%%%%%%%%%%%%%%%%%%%%%%
We now make a finer analysis. We have three cases:
\begin{enumerate}
\item[1)] $\overline x^{\overline c}\overline y$ is the unique minimal value term in the expansion (\ref{eq6}).
\item[2)] $\overline x^{\overline c}\overline y$ and $\overline y^p$ are  the two minimal value terms in the expansion (\ref{eq6}).\item[3)] $\overline y^p$ is the unique minimal value term in the expansion (\ref{eq6}).
\end{enumerate}

Suppose that we are in Case 1), so that $\overline x^{\overline c}\overline y$ is the unique minimal value term in the expansion (\ref{eq6}), so $\beta_1=1$. Further, $\Lambda=\gamma_1+x_1\Omega$ (with $\gamma_1\in k\ne 0$). Now $\overline \sigma(d-c)=\tau=\beta_1=1$, and $\overline \sigma(f-e)=m$. Thus $\overline \sigma=1$, $d-c=1$ and $f-e=m$. Thus
$$
\begin{array}{lll}
u_1&=&x_1(y_1+\alpha)^{c}\Lambda^e\\
v_1&=&(y_1+\alpha)(\gamma_1+x_1\Omega)^{m}-\alpha\gamma_1^{m}=\gamma_1y_1+x_1\Omega'.
\end{array}
$$
Thus $R_1\rightarrow S_1$ is unramifed and we are in Case 0) of the conclusions of the theorem. 

Suppose that we are in Case 2), so that 
 $\overline x^{\overline c}\overline y$ and $\overline y^p$ are the two  minimal value terms in the expansion (\ref{eq6}).
  The expansion of (\ref{eq6}) is then
$$
\overline v=\gamma_1x^{\overline c}\overline y+\gamma_2 \overline y^p+\mbox{higher value terms}.
$$
Since $\mu(\overline y^p)=\mu(\overline x^{\overline c}\overline y)$ we have that $(p-1)\mu(\overline y)=\overline c\mu(\overline x)$. Thus $\mbox{char }k\ne 2$ since $\mu(\overline y)\not\in \mu(\overline x)\ZZ$. Let $\psi=\mbox{gcd}(p-1,\overline c)$. Then in the substitution
$$
\overline x=x_1^m(y_1+\alpha)^{a'}, \overline y=x_1^q(y_1+\alpha)^{b'}
$$
with $mb'-a'q=1$ of (\ref{eq5}), we have $\psi m=p-1$ and $\psi q=\overline c$.
Substituting (\ref{eq5}) in $u$ and $\overline v$, we obtain
$$
u=x_1^m(y_1+\alpha)^{a'}
$$
and
$$
\overline v= \gamma_1x_1^{\overline cm+q}(y_1+\alpha)^{a'\overline c+b'}+\gamma_2 x_1^{qp}(y_1+\alpha)^{b'p}+\cdots= x_1^{qp}(y_1+\alpha)^{a'\overline c+b'}\Lambda
$$
where $\Lambda=\gamma_2(y_1+\alpha)^{\psi}+\gamma_1+x_1\Omega$. 

Suppose that $\Lambda$ is not a unit. 
Let $d_1=\mbox{ord}_{y_1}\left[(y_1+\alpha)^{a'\overline c+b'}\Lambda(0,y_1)\right]$. We have that $0<d_1<\infty$ since $\Lambda$ is not a unit. By \cite[Proposition 3.1]{C1} and since our extension is Galois, we have that the complexity $md_1$ of $R\rightarrow S_1$ divides $p=[L:K]$,
which is a contradiction to our assumption that $m>1$ and the fact that $m$ divides $p-1$.

Suppose  that $\Lambda$ is a unit. Following the analysis of the case when $\Lambda$ is a unit above, we have that $\tau=\beta_1=1$ and $\overline\sigma=\mbox{gcd}(m,\overline cm+q)=1$.
Thus from (\ref{eq3*}), we have $d-c=1$ and from (\ref{eq9}), we have $f-e=m$. From (\ref{eq10}), we obtain
$$
u_1=x_1(y_1+\alpha)^c\Lambda^e
$$
and
$$
v_1=(y_1+\alpha)\Lambda^m-\alpha(\gamma_2\alpha^{\psi}+\gamma_1)^m.
$$
We compute
$$
v_1(0,y_1)=(y_1+\alpha)(\gamma_2(y_1+\alpha)^{\psi}+\gamma_1)^m-\alpha(\gamma_2\alpha^{\psi}+\gamma_1)^m.
$$
We have
$$
\begin{array}{lll}
\frac{\partial}{\partial y_1}v_1(0,y_1)&=&   (\gamma_2(y_1+\alpha)^{\psi}+\gamma_1)^m+(y_1+\alpha)m(\gamma_2(y_1+\alpha)^{\psi}+\gamma_1)^{m-1}\psi\gamma_2 \psi(y_1+\alpha)^{\psi-1}\\
&=& (\gamma_2(y_1+\alpha)^{\psi}+\gamma_1)^{m-1}[\gamma_2(y_1+\alpha)^{\psi}+\gamma_1+m(y_1+\alpha)\psi\gamma_2(y_1+\alpha)^{\psi-1}]\\
&=& (\gamma_2(y_1+\alpha)^{\psi}+\gamma_1)^{m-1}[\gamma_2(y_1+\alpha)^{\psi}+\gamma_1+\gamma_2(p-1)(y_1+\alpha)^{\psi-1}]\\
&=&(\gamma_2(y_1+\alpha)^{\psi}+\gamma_1)^{m-1}\gamma_1
\end{array}
$$
is a unit so $\mbox{ord}_{y_1}v_1(0,y_1)=1$. Thus the complexity of $R_1\rightarrow S_1$ is $m<p$,  so the complexity must be one, so that $R_1\rightarrow S_1$ is unramified and we are in Case 0) of the conclusions of the theorem.

Now suppose that we are in Case 3) so that $\overline y^p$ is the unique minimal value term of the expansion (\ref{eq6}) of $\overline v$. Then $\beta_1=p$ in (\ref{eq6}) and $\Lambda=\gamma_1+x_1\Omega$ is a unit in (\ref{eq20}). In the analysis of the case when $\Lambda$ is  a unit following (\ref{eq20}), we have that 
$d-c=1$ if $\overline\sigma=p$ and $d-c=p$ if $\overline\sigma=1$, so that 
$$
\mbox{ord}_{y_1}v_1(0,y_1)=\mbox{ord}_{y_1}(y_1+\alpha)^{d-c}\gamma_1^{f-e}-\alpha^{d-c}\gamma_1^{f-e}
=\left\{\begin{array}{ll}
1&\mbox{ if }\overline\sigma=p\\
p &\mbox{ if }\overline \sigma=1.
\end{array}\right.
$$
Thus from (\ref{eq10}), we see that $R_1\rightarrow S_1$ is of type 1
 if $\overline\sigma=p$ and is of type 2 if $\overline\sigma=1$.

%%%%%%%%%%%%%%%%%%%%%

We now establish that if Case 3) above holds ($\overline y^p$ is the unique minimal value term in (\ref{eq6})), then the invariant $\overline \sigma=\mbox{gcd}(m,\alpha_1m+\beta_1q)$, defined in (\ref{eq11}), is such that $\overline\sigma=\mbox{gcd}(m,pq)$, so that the $\sigma$ defined in the statement of the theorem is $\overline\sigma$. Since $\overline y^p$ is the unique minimal value term in (\ref{eq6}), then 
$\alpha_1=0$ and $\beta_1=p$, giving the desired equality. 

In our analysis above, we saw that if $\mu(\overline x^{\overline c}\overline y)\le \mu (\overline y^p)$ then $R_1\rightarrow S_1$ is in Case 0) of the conclusions of the theorem and if $\mu(\overline x^{\overline c}\overline y)> \mu (\overline y^p)$ then $R_1\rightarrow S_1$ is in Case 1) or Case 2) of the conclusions of the theorem. Since $\mu(\overline x^{\overline c}\overline y)=(m\overline c+q)$ and $\mu(\overline y^p)=pq$, we have that if  $\frac{q}{m}\ge  \frac{\overline c}{p-1}$ then $R_1\rightarrow S_1$ is unramifed and if $\frac{q}{m}< \frac{\overline c}{p-1}$
then $R_1\rightarrow S_1$ is either of type 1 or type 2.

%%%%%%%%%%%%%%%%%%%%%%%%%%%%%%%%%
We now establish the formulas for the Jacobian ideal $J(S_1/R_1)$.
We have that 
$$
J(R_{1}/R)J(S_{1}/R_{1})=J(S/R)J(S_{1}/S),
$$
 where
$J(S_{1}/R_{1})=x_{1}^{c_{1}}S_1$, $J(S/R)=\overline x^{\overline c}S$ and $J(S_{1}/S)=x_{1}^{m+q-1}S_1$.

We have that  $\sigma=\mbox{gcd}(m,pq)$ which is 1 or $p$.
Thus $m=\sigma \overline m$ and $pq=\sigma\overline q$. Further, we have shown that $\sigma=1$ implies $R_{1}\rightarrow S_{1}$ is of type 1 and $\sigma=p$ implies $R_{1}\rightarrow S_{1}$ is of type 2.

Now $\sigma=1$ implies $p$ divides $\overline q$ and $\sigma=p$ implies $p$ does not divide $\overline q$, since $p$ divides $m$ implies $q\not| q=\overline q$. We have that $J(R_{1}/R)=u_{1}^{\overline m+\overline q-1}R_1$ and thus
$$
J(R_{1}/R)S_{1}=\left\{
\begin{array}{ll}
x_{1}^{m+pq-1}S_1&\mbox{ if }\sigma=1\\
x_{1}^{m+pq-p}S_1&\mbox{ if }\sigma=p.
\end{array}\right.
$$
In the case $\sigma=1$ we have 
$(x_{1}^{m+pq-1})(x_{1}^{c_{1}})S_1=(\overline x^{\overline c})(x_{1}^{m+q-1})S_1$ so 
$(x_{1}^{c_{1}})S_1=(x_{1}^{m\overline c-(p-1)q})S_1$ and we obtain the formula of Case 1) of the conclusions of the theorem. 
In the case $\sigma=p$ we have 
$(x_{1}^{m+pq-p})(x_{1}^{c_{1}})S_1=(\overline x^{\overline c})(x_{1}^{m+q-1})S_1$ so 
$(x_{1}^{c_{1}})S_1=(x_{1}^{m\overline c-(p-1)q+(p-1)})S_1$ and we obtain the formula of Case  2) of the conclusions of the theorem.
\end{proof}

\begin{Remark}\label{RemarkA} Suppose that $\omega$ is a rational rank 1 nondiscrete valuation of $L$ dominating $S$ and $R\rightarrow S$ is of type 1. Let $\nu$ be the restriction of $\omega$ to $K$. Let $\overline x=u$ and  $\overline y$ be the difference of $y$ and a nonzero polynomial in $\overline x$ so that $\omega(\overline y)\not\in \omega(\overline x)\ZZ$. Let $\overline v$ be the change of variables in Theorem \ref{TheoremA}. 

Define $m$ and $q$ to be the unique relatively prime positive integers such that $m\omega(\overline y)=q\omega(x)$. We have that $m>0$. There exist $0\ne \alpha\in k$ and $a',b'\in \NN$ such that $mb'-qa'=1$ and if $S\rightarrow S_1$ is the sequence of quadratic transforms defined by 
$$
\overline x=x_1^m(y_1+\alpha)^{a'}, \overline y=x_1^q(y_1+\alpha)^{b'}
$$
then $\omega$ dominates $S_1$.

Let $\nu$ be the restriction of $\omega$ to $K$. The formulas of Cases 0), 1) and 2) of Theorem \ref{TheoremA} can then be stated in terms of the valuation $\omega$. They are:
\begin{enumerate}
\item[0)] If $\frac{q}{m}\ge  \frac{\overline c}{p-1}$  then $R_1\rightarrow S_1$ is unramified.\item[1)] If $\frac{q}{m}< \frac{\overline c}{p-1}$ and $\sigma=1$ then $R_1\rightarrow S_1$ is of type 1 and
$$
\left(\frac{c_{1}}{p-1}\right)\omega(x_{1})=\left(\frac{\overline c}{p-1}\right)\omega(x)-\omega(\overline y).
$$  
\item[2)] If $\frac{q}{m}< \frac{\overline c}{p-1}$ and $\sigma=p$ then $R_1\rightarrow S_1$ is of type 2 and 
$$
\left(\frac{c_{1}}{p-1}\right)\omega(x_{1})=\left(\frac{\overline c}{p-1}\right)\omega(x)-\omega(\overline y)+\omega(x_{1}).
$$
\end{enumerate}
In the conclusions of the theorem, suppose that $R_1\rightarrow S_1$ is of type 1. Then we necessarily have that $\nu(\overline v)\not\in \nu(u)\ZZ$ since $\sigma=1$ and thus $\overline m=m>1$. 

\end{Remark}

\begin{Theorem}\label{TheoremB} Suppose that $R\rightarrow S$  is of type 2 with respect to regular parameters $x,y$ in $S$ and $u,v$ in $R$ and that $J(S/R)=(x^{\overline c})$.  Let $g(u)\in k[u]$ be a polynomial with no constant term. Make the change of variables, letting $\overline v=v-g(u)$ and $\overline y=\overline v$, so that $x,\overline y$ are regular parameters in $S$ and $u,\overline v$ are regular parameters in $R$.

Suppose that $m,q$ are positive integers with $\mbox{gcd}(m,q)=1$. Let $\alpha$ be a nonzero element of $k$.
Consider the sequence of quadratic transforms $S\rightarrow S_1$ so that $S_1$ has regular parameters $x_1,y_1$ defined by
$$
x=x_1^{m}(y_1+\alpha)^{a'}, \overline y=x_1^{q}(y_1+\alpha)^{b'}
$$
where  $a',b'\in \NN$ are such that $mb'-q a'=1$. 

Let $\sigma=\mbox{gcd}(pm,q)$ which is 1 or $p$. There exists a unique sequence of quadratic transforms $R\rightarrow R_1$ such that $R_1$ has regular parameters $u_1,v_1$ defined by 
$$
u=u_1^{\overline m}(v_1+\beta)^{c'}, \overline v=u_1^{\overline q}(v_1+\beta)^{d'}
$$
where $pm=\sigma\overline m$, $q=\sigma \overline q$, $\overline md'-c'\overline q=1$ and $0\ne\beta\in k$,
 giving
  a commutative diagram of homomorphisms
$$
\begin{array}{lll}
R_1&\rightarrow &S_1\\
\uparrow&&\uparrow\\
R&\rightarrow &S
\end{array}
$$
such that $R_1\rightarrow S_1$ is quasi finite. 
We have that  $J(S_1/R_1)=(x_1^{c_1})$ for some positive integer $c_1$.   Further: 
\begin{enumerate}
%\item[0)] If $\frac{q}{m}\ge \frac{\overline c}{p-1}?$ then $R_1\rightarrow S_1$ is unramified.
\item[1)] If  $\sigma=1$ then $R_1\rightarrow S_1$ is of type 1 and
$$
\left(\frac{c_{1}}{p-1}\right)=\left(\frac{\overline c}{p-1}\right)m-m.
$$  
\item[2)] If  $\sigma=p$ then $R_1\rightarrow S_1$ is of type 2 and 
$$
\left(\frac{c_{1}}{p-1}\right)=\left(\frac{\overline c}{p-1}\right)m-m+1.
$$
\end{enumerate}

\end{Theorem}

\begin{proof}
Substitute
 \begin{equation}
x=x_1^{m}(y_1+\alpha)^{a'}, \overline y=x_1^{q}(y_1+\alpha)^{b'}.
\end{equation}
into $u$ and  $\overline v$ to obtain
$$
\begin{array}{lll}
u&=& x_1^{mp}(y_1+\alpha)^{a'p}(\lambda+x_1\Omega)\\
\overline v&=&x_1^{q}(y_1+\alpha)^{b'}
\end{array}
$$
where $0\ne \lambda\in k$.
 Let
$$
\sigma=\mbox{gcd}(mp,q).
$$
Let 
\begin{equation}
\tau=\mbox{Det}\left(\begin{array}{cc}
mp&a'p\\
q&b'\end{array}\right)
=p(mb'-a'q)=p.
\end{equation}
Let $R\rightarrow R^*$ be the sequence of quadratic transforms defined by $u=u_1^{g}\overline v_1^{g'}$, $\overline v=u_1^{h}\overline v_1^{h'}$ where $g,g',h,h'\in \NN$, $gh'-hg'=\pm 1$ and 
$$
u_1=x_1^{\sigma}(y_1+\alpha)^c(\lambda+x_1\Omega)^e,
\overline v_1=x_1^{\sigma}(y_1+\alpha)^d(\lambda+x_1\Omega)^f
$$
where 
\begin{equation}
\mbox{Det}\left(\begin{array}{cc}
\sigma &c\\
\sigma&d
\end{array}\right)
=\sigma(d-c)=\tau=p
\end{equation}
Now perform a single quadratic transform $R^*\rightarrow R_1$ so that $R_1$ is dominated by $S_1$ and 
$R_1$ has regular parameters $u_1,v_1$ satisfying 
\begin{equation}
u_1=x_1^{\sigma}(y_1+\alpha)^{c}(\lambda+x_1\Omega)^e,
v_1=\frac{\overline v_1}{u_1}-\alpha^{d-c}\lambda^{f-e}=(y_1+\alpha)^{d-c}(\lambda+x_1\Omega)^{f-e}-\alpha^{d-c}\lambda^{f-e}.
\end{equation}
We have an expression for $R\rightarrow R_1$ of the form
$$
u=u_1^{\overline m}(v_1+\beta)^{\overline a_1'}, \overline v=u_1^{\overline q}(v_1+\beta)^{\overline b_1'}
$$
where $\overline m=\frac{mp}{\sigma}$, $\overline q=\frac{q}{\sigma}$ and $\beta=\alpha^{\frac{p}{\sigma}}\gamma(0,0)^{-\frac{q}{\sigma}}$.
We have that $v_1(0,y_1)=(y_1+\alpha)^{d-c}\lambda^{f-e}-\alpha^{d-c}\lambda^{f-e}$ with $d-c>0$.
Hence
$$
0<d_1=\mbox{ord}_{y_1}v_1(0,y_1)<\infty.
$$
If $\sigma=p$ then $d-c=1$ so that $d_1=1$ and the complexity of $R_1\rightarrow S_1=p$. We then have that $R_1\rightarrow S_1$ is of type 2, so that it is in Case 2 of the conclusions of the theorem. If $\sigma=1$ then $d-c=p$ and $d_1=p$ so that the complexity of $R_1\rightarrow S_1$ is $p$ and it is in Case 1 of the conclusions of the theorem.

If  $\sigma=1$ then  $p$ divides $\overline m$ and $\sigma=p$ implies $p$ does not divide $\overline m$, since $p$ then divides $q$ which implies $p$ does not divide $m=\overline m$.

 We have that $J(R_{1}/R)=(u_{1}^{\overline m+\overline q-1})$ and thus
$$
J(R_{1}/R)S_{1}=\left\{
\begin{array}{ll}
x_{1}^{pm+q-1}S_1&\mbox{ if }\sigma=1\\
x_{1}^{pm+q-p}S_1&\mbox{ if }\sigma=p.
\end{array}\right.
$$
In the case $\sigma=1$ we have 
$(x_{1}^{pm+q-1})(x_{1}^{c_{1}})S_1=(x^{\overline c})(x_{1}^{m+q-1})S_1$ so 
$(x_{1}^{c_{1}})S_1=(x_{1}^{m\overline c-m(p-1)})S_1$ and we obtain the formula of Case 1) of the statement of the theorem.

In the case $\sigma=p$ we have 
$(x_{1}^{pm+q-p})(x_{1}^{c_{1}})S_1=(x^{\overline c})(x_{1}^{m+q-1})S_1$ so 
$(x_{1}^{c_{1}})S_1=(x_{1}^{m\overline c-(p-1)m+(p-1)})S_1$ and we obtain the formula of Case 2) of the statement of the theorem.
\end{proof}

\begin{Remark}\label{RemarkB} Suppose that $\omega$ is a nondiscrete rational rank 1 valuation of $L$ dominating $S$ and $R\rightarrow S$ is of type 2.   Let $\nu$ be the restriction of $\omega$ to $K$. 
   Make the change of variables, letting $\overline v$ be the difference of $v$ and a polynomial in $u$ so that $\omega(\overline v)\not\in \omega(u)\ZZ$ and letting $\overline y=\overline v$.

Define $m$ and $q$ to be the unique relatively prime positive integers such that $m\omega(\overline y)=q\omega(x)$.  There exist $0\ne \alpha\in k$ and $a',b'\in \NN$ such that $mb'-qa'=1$ and if $S\rightarrow S_1$ is the sequence of quadratic transforms defined by 
$$
x=x_1^m(y_1+\alpha)^{a'}, \overline y=x_1^q(y_1+\alpha)^{b'}
$$
then $\omega$ dominates $S_1$.

The formulas of Cases 1) and 2) of Theorem \ref{TheoremB} can then be stated in terms of the valuation $\omega$. They are:
\begin{enumerate}
\item[1)] If  $\sigma=1$ then $R_1\rightarrow S_1$ is of type 1 and
$$
\left(\frac{c_{1}}{p-1}\right)\omega(x_{1})=\left(\frac{\overline c}{p-1}\right)\omega(x)-\omega(x).
$$  
\item[2)] If  $\sigma=p$ then $R_1\rightarrow S_1$ is of type 2 and 
$$
\left(\frac{c_{1}}{p-1}\right)\omega(x_{1})=\left(\frac{\overline c}{p-1}\right)\omega(x)-\omega(x)+\omega(x_{1}).
$$
\end{enumerate}

If $R_1\rightarrow S_1$ is of type 2, then 
we  have that $\nu(\overline y)\not\in \nu(x)\ZZ$, since $\mbox{gcd}(pm,q)=p$. 
\end{Remark}

We will show that $\nu(\overline y)\not\in \nu(x)\ZZ$ if $R_1\rightarrow S_1$ is of type 2.  We have that $\mbox{gcd}(pm,q)=p$.
If $\nu(\overline y)\in\nu(x)\ZZ$, then  $\nu(\overline y)=q\nu(x)$ and since $p\mid q$, we  have that  $\nu(\overline v)=\nu(u^{\overline q})$  with $\overline q=\frac{q}{p}$, a contradiction to the assumption that $\nu(\overline v)\not\in \nu(u)\ZZ$.

\section{Switching of types of extensions under blowing up}\label{SecSwitch}

Suppose that $K$ and $L$ are two dimensional algebraic function fields over an algebraically closed field $k$ of characteristic $p>0$ and $K\rightarrow L$ is an Artin-Schreier extension. Suppose that $R$ is a regular algebraic local ring of $K$ and $S$ is a regular algebraic local ring of $S$ such that $S$ dominates $R$ and $R\rightarrow S$ is of type 1 or 2 as defined at the beginning of Section \ref{SecCalc}. Further assume that the Jacobian ideal $J(S/R)$ satisfies $\sqrt{J(S/R)}=xS$. Let $P=uR$ and $Q=xS$. Then $R\rightarrow S$ is well prepared and  the regular parameters $u,v$ and $x,y$ are admissible parameters. Such an extension $R\rightarrow S$ exists by Remark \ref{RemarkN1} and Proposition \ref{Prop1*}.

Inductively applying Theorems \ref{TheoremA} and  \ref{TheoremB},  and making choices for the construction of $S_i\rightarrow S_{i+1}$ consistent with the assumptions of Theorems \ref{TheoremA} and \ref{TheoremB},
we construct a diagram where the horizontal sequences are birational extensions of regular local rings (sequences of quadratic transforms)
\begin{equation}\label{eq30}
\begin{array}{ccccccccc}
S&=&S_0&\rightarrow &S_1&\rightarrow & S_2&\rightarrow & \cdots\\
&&\uparrow&&\uparrow&&\uparrow&&\\
R&=&R_0&\rightarrow &R_1&\rightarrow & R_2&\rightarrow & \cdots
\end{array}
\end{equation}

Each $R_i\rightarrow S_i$ is well prepared of type 0, 1 or 2.

Each $R_i$ has allowable regular parameters $(u_i,v_i)$ and $(u_i,\overline v_i)$ and each $S_i$ has allowable regular parameters $(x_i,y_i)$ and $(\overline x_i,\overline y_i)$. The map $S_i\rightarrow S_{i+1}$ is defined by
\begin{equation}\label{eqN1}
\overline x_i=x_{i+1}^{m_{i+1}}(y_{i+1}+\alpha_{i+1})^{a_{i+1}'},
\overline y_i=x_{i+1}^{q_{i+1}}(y_{i+1}+\alpha_{i+1})^{b_{i+1}'}
\end{equation}
and the map $R_i\rightarrow R_{i+1}$ is defined by
\begin{equation}\label{eqN2}
u_i=u_{i+1}^{\overline m_{i+1}}(v_{i+1}+\beta_{i+1})^{c_{i+1}'},
\overline v_i=u_{i+1}^{\overline q_{i+1}}(v_{i+1}+\beta_{i+1})^{d_{i+1}'}
\end{equation}
where $0\ne \alpha_{i+1}, 0\ne \beta_{i+1}\in k$. 
%If $R_i\rightarrow S_i$ is of type 1, then we set $\overline u_i=x_i$. If $R_i\rightarrow S_i$ is of type 2, we set $\overline u_i=u_i$.
If $R_i\rightarrow S_i$ is of type 1 or of type 2 then $\overline x_i$, $\overline y_i$ and $\overline v_i$ are defined by our changes of variables in Theorem \ref{TheoremA} or \ref{TheoremB}.   
%We make changes of variables, setting $\overline v_i=v_i-g(x)$ where $g(x)\in k[x]$ is a polynomial with zero constant term and setting $\overline y_i=\overline v_i$. 
If  $R_i\rightarrow S_i$ is of type 0, then we take $\overline x_i=u_i$ and $\overline y_i=\overline v_i=v_i$.
 %We then choose $m_{i+1}=\overline m_{i+1}$ and $q_{i+1}=\overline q_{i+1}$ so that they are relatively prime, choose $a_{i+1}',b_{i+1}'\in \NN$ so that $m_{i+1} b'_{i+1}-q_{i+1} a'_{i+1}=1$ and choose $0\ne \alpha\in k$. In this case $R_{i+1}\rightarrow S_{i+1}$ is again of type 0, and $R_j\rightarrow S_j$ will be of type 0 for $j\ge i$.
If $R_i\rightarrow S_i$ is of type 2, we will impose the extra condition that 
\begin{equation}\label{eqN3}
\overline m_{i+1}=\frac{pm_{i+1}}{\mbox{gcd}(pm_{i+1},a_{i+1})}>1.
\end{equation}

We will say that the sequence (\ref{eq30}) switches infinitely often if there are infinitely many $i$ such that $R_i\rightarrow S_i$ is of type 1 and there are infinitely may $i$ such that $R_i\rightarrow S_i$ is of type 2.

Since $\mbox{trdeg}_{K}=2$,  we have that $\cup_{i=1}^{\infty} R_i$ and $\cup_{i=1}^{\infty}S_i$ are valuation rings (by \cite[Lemma12]{Ab1}). Further, given $f\in R_i$ (or $f\in S_i$), there exists $j\ge i$ such that there is an expression $f=u_i^{t_i}\gamma_i$ where $t_i\in \NN$ and $\gamma_i$ is a unit in $R_j$ (or $f=x_i^{t_i}\gamma_i$ where $t_i\in \NN$ and $\gamma_i$ is a unit in $S_j$), as shown for instance in \cite{Ab1}.

Let $\nu$ and $\omega$ be valuations which have these respective valuation rings and such that $\omega|K=\nu$. These valuations are uniquely determined up to equivalence of valuations.
We have that $\omega$ and $\nu$ are  nondiscrete rational rank 1 valuations, with  value groups
$$
\nu K=\cup_{i=1}^{\infty}\frac{1}{\overline m_1\overline m_2\cdots \overline m_i}\ZZ\nu(u)\mbox{ and }
\omega L=\cup_{i=1}^{\infty}\frac{1}{ m_1 m_2\cdots  m_i}\ZZ\nu(x).
$$

Equation (\ref{eqN3}) is just the statement that $\nu(\overline v_{i})\not \in \ZZ\nu(u_{i})$. The condition that all $m_{i+1}>1$  in Theorem \ref{TheoremA} is just the statement that $\omega(\overline y_{i})\not\in \ZZ\omega(\overline x_{i})$.

Suppose that  $\overline \omega$ is a valuation ring of $L$ which dominates $S$ which is nondiscrete of rational rank 1 and $\overline \nu=\overline \omega|K$.  Then we can inductively construct 
a sequence (\ref{eq30}) so that $\omega$ dominates $S_i$ for all $i$, and so $\mathcal O_{\overline \nu}=\cup_{i=1}^{\infty} R_i$ and $\mathcal O_{\overline \omega}=\cup_{i=1}^{\infty}S_i$, so that the valuations $\omega$ and $\nu$ determined by the sequence are $\overline\omega$ and $\overline \nu$ respectively (up to equivalence of valuations).

The complexity of the maps $R_i\rightarrow S_i$ in the diagram (\ref{eq30})  must  either be $p$ for  all $i$, or will be $p$ until some $i_0$ and then the complexity will be $1$ for all $j\ge i_0$, so that $R_j\rightarrow S_j$ is of type 1 or 2 for $j<i_0$ and $R_j\rightarrow S_j$ is unramified (has type 0) for all $j\ge i_0$.

 We will say  a sequence (\ref{eq30}) has stable complexity $p$ if the complexity of $R_i\rightarrow S_i$ is $p$ for all $i\ge 0$.
With this assumption,  each map $R_i\rightarrow S_i$ in (\ref{eq30}) is either of type 1 or of type 2. We draw the following conclusions from Theorems \ref{TheoremA} and \ref{TheoremB}.

Assume that the stable complexity of a sequence (\ref{eq30}) is $p$. If $R_r\rightarrow S_r$ is of type 1, then $S_r\rightarrow S_{r+1}$ is the standard sequence of quadratic transforms along $\omega$. Further, $R_r\rightarrow R_{r+1}$ is the standard sequence of quadratic transforms along $\nu$ unless $m_{r+1}=p$. In this case, $\overline m_{r+1}=1$, so that the standard sequence of quadratic transforms of $R_r$ along $\nu$ dominates $R_{r+1}$, and $R_{r+1}\rightarrow S_{r+1}$ is of type 2.

If $R_r\rightarrow S_r$ is of type 2, then $R_r\rightarrow R_{r+1}$ is the standard sequence of quadratic transforms along $\nu$. Further, $S_r\rightarrow S_{r+1}$ is the standard sequence of quadratic transforms along $\omega$ unless $\overline m_{r+1}=p$. In this case, $m_{r+1}=1$, so that the standard sequence of quadratic transforms of $S_r$ along $\omega$ dominates $S_{r+1}$, and $R_{r+1}\rightarrow S_{r+1}$ is of type 1.

\begin{Proposition} \label{Prop2} Suppose that a sequence (\ref{eq30}) has stable complexity $p$. Then the sequence (\ref{eq30}) switches infinitely often if and only if $vK$ is $p$-divisible.
\end{Proposition}

\begin{proof} We have that  $\omega L$ is $p$-divisible if and only if $\nu K$ is $p$-divisible (for instance by (3) of \cite[Lemma 7.32]{CP}).

Suppose that $\nu K$ is not $p$-divisible. Then there exists $r_0$ such that for $r>r_0$ we have that $p\not| \overline m_r$ and $p\not| m_r$. Now if $r>r_0$ and $R_r\rightarrow S_r$ is of type  1 then $R_{r+1}\rightarrow S_{r+1}$ must be of type 1 since $\sigma_{r+1}=p$ in Theorem \ref{TheoremA} implies $p$ divides $m_{r+1}$. Further, if $R_r\rightarrow S_r$ is of type 2 then $R_{r+1}\rightarrow S_{r+1}$ must be of type 2 since $\sigma_{r+1}=1$ in Theorem \ref{TheoremB} implies $p$ divides $\overline m_{r+1}$. Thus for $r>r_0$ there can be no switching. 

Suppose that $\nu K$ is $p$-divisible. Suppose that (\ref{eq30})  doesn't switch infinitely often. Then there exists $r_0$ such that for $r\ge r_0$ $R_r \rightarrow S_r$ is of the same type as $R_{r_0}\rightarrow S_{r_0}$. 
Suppose that 
$R_{r_0}\rightarrow S_{r_0}$ is of type 1. Since $\omega L$ is $p$-divisible, there exists $r\ge r_0$ such that $p$ divides $m_{r+1}$. 
But then $R_{r+1} \rightarrow S_{r+1}$ must be of type 2 since $\sigma_{r+1}=p$ in Theorem \ref{TheoremA}.
Suppose that 
$R_{r_0}\rightarrow S_{r_0}$ is of type 2. Since $\nu L$ is $p$-divisible, there exists $r\ge r_0$ such that $p$ divides $ m_{r+1}$. Now $p$ divides $m_{r+1}$ implies $p\not| q_{r+1}$ which implies $\sigma_{r+1}=\mbox{gcd}(pm_{r+1},q_{r+1})=1$ in Theorem \ref{TheoremB}.
But then $R_{r+1} \rightarrow S_{r+1}$ must be of type 1.
\end{proof}

\begin{Remark} Suppose that a sequence (\ref{eq30})has stable complexity $p$. If $\nu K$ is not $p$-divisible then  $m_r=\overline m_r$ for $r\gg 0$ and we have that 
$m_r>1$ and $\overline m_r>1$ for $r\gg 0$ in (\ref{eq30}).
\end{Remark}

As the following Proposition shows, the nicest form that a sequence (\ref{eq30}) can take is when $R_r\rightarrow S_r$ is of  type 2 for all $r\gg 0$. This is the strongly monomial form (defined in the introduction).

\begin{Proposition}\label{PropositionSB}
  The following are equivalent for a sequence (\ref{eq30}) with stable complexity $p$, and  valuations $\nu$ and $\omega$ which it determines.
\begin{enumerate}
\item[1)] There exists an $r_0$ such that $R_r\rightarrow S_r$ is of type 2 in (\ref{eq30}) for $r\ge r_0$. 
%\item[2)] Strong monomialization holds in the valued extension $L/K$.
\item[2)] $[\omega L:\omega K]=p$.
\item[3)] The valued extension $L/K$ is defectless.
\end{enumerate}
\end{Proposition}

\begin{proof} Since the  stable complexity of the sequence (\ref{eq30})  is $p$, we have that 
$$
p=[\omega L:\omega K]\delta(\omega/\nu)
$$
 by Remark \ref{RemarkAS} and Proposition \ref{Prop2*}. Thus statement 2) is equivalent to statement 3). We now prove that statement 1) is equivalent to statement 2). Suppose that there exists $r_0$ such that $R_r\rightarrow S_r$ is of type 2 in (\ref{eq30}) for $r\ge r_0$. Then $u_i=\gamma_i\overline x_i^p$ for all $i\ge r_0$ and $m_i=\overline m_i$ for $i\ge r_0$ by Theorem \ref{TheoremB}.
Thus 
$$
\nu K=\cup_{i=r_0+1}^{\infty}\frac{1}{\overline m_{r_0+1}\cdots \overline m_i}\ZZ\nu(u_{r_0})
=\cup_{i=r_0+1}^{\infty}\frac{1}{m_{r_0+1}\cdots m_i}\ZZ p\omega(x_{r_0})=p\omega L.
$$
If $\nu K=\omega L$ then $p\omega L=\omega L$ which implies that $\omega L$ is $p$-divisible, a contradiction to Proposition \ref{Prop2}. Thus $[\omega L:\nu K]=p$. 

Suppose that there exists $r_0$ such that $R_r\rightarrow S_r$ is of type 1 in (\ref{eq30}) for $r\ge r_0$. Then $u_i=x_i$ for $i\ge r_0$. Thus
$$
\nu K=\cup_{i=r_0}^{\infty}\ZZ\nu(u_i)=\cup_{i=r_0}^{\infty}\ZZ\omega(x_i)=\omega L.
$$
Finally, suppose that (\ref{eq30}) switches infinitely often. Then $\nu K$ and $\nu L$ are $p$-divisible by Proposition \ref{Prop2}. Since $[L:K]=p$, $\omega L= p\omega L\subset \nu K\subset \omega L$ which implies that $\nu K=\omega L$.

\end{proof}

%\begin{Proposition}\label{Prop1}
%Suppose that the valued field extension $L/K$ satisfies the equivalent conditions of Remark \ref{RemarkSB}. Then
% the ramification ideals of $L/K$ satisfy 
%$$
%I_G=I_G'=(x_r^{\frac{c_r}{p-1}-1}) 
%$$
 %and 
% the cut
 %$\nu( \Theta -K )=-(\frac{c_r}{p-1}-1)\nu(x_r)$ for all $r\ge r_0$ in (\ref{eq30}).
 %\end{Proposition}
 
% \begin{proof} From Theorem \ref{TheoremB} and \ref{RemarkB}, we have that 
 %$$
 %\omega(\frac{1}{x_{r_0+i}}J(S_{r_0+i}/R_{r_0+i}))=\omega(\frac{1}{x_{r_0}}J(S_{r_0}/R_{r_0})
 %$$
 %for all $i\ge 1$. We then have that 
 %$$
 %\begin{array}{lll}
 %\mbox{inf}\{\omega(J(S_{r_0+i}/R_{r_0+i}))\}
 %&=&\mbox{inf}\left\{(\left(\frac{c_{r_0}}{p-1}\right)\omega(x_{r_0})-\omega(x_{r_0})+m_{r_0}\cdots m_{r_0+i}\omega(x_{r_0})\right\}\\
 %&=&(\left(\frac{c_{r_0}}{p-1}\right)\omega(x_{r_0})-\omega(x_{r_0}).
 %\end{array}
 %$$
 % \end{proof}

We see that if  the sequence (\ref{eq30}) switches infinitely often then the  extension must be a defect extension. 
Any configuration of switching is possible. A sequence with prescribed switching can be created by iterating the constructions of Theorems \ref{TheoremA} and \ref{TheoremB}.

If a sequence stabilizes with $R_r\rightarrow S_r$ of type 2 for all $r\ge r_0$, then 
from iteration of formula 2) of Theorem \ref{TheoremB}, for all $s>0$ we have that
$$
\left(\frac{c_{r_0+s}}{p-1}\right)\frac{1}{m_1\cdots\cdot m_{r_0+s}}
=\left(\frac{c_{r_0}}{p-1}\right)\frac{1}{m_1\cdots m_{r_0}}-\frac{1}{m_1\cdots m_{r_0}}+\frac{1}{m_1\cdots+m_{r_0+s}}.
$$
By Remarks \ref{RemarkA} and \ref{RemarkB}, $\omega(J(S_i/R_i))=c_i\omega(x_i)$ is monotonically decreasing with $i$. We calculate the for $s\ge 1$,
$$
\omega(J(S_{r_0+s}/R_{r_0+s}))=\left[\frac{c_{r_0}}{m_1\cdots m_{r_0}}-\frac{(p-1)}{m_1\cdots m_{r_0}}+\frac{p-1}{m_1\cdots m_{r_0+s}}\right]\omega(x_0).
$$
Thus since infinitely many $m_i$ are greater than 1,
$$
\inf\{\omega(S_i/R_i)\}= \left[\frac{c_{r_0}}{m_1\cdots m_{r_0}}-\frac{(p-1)}{m_1\cdots m_{r_0}}\right]\omega(x_0)\in \omega(L).
$$ 
Thus by Proposition \ref{Prop100},
$$
{\rm dist}(\omega/\nu)=-\frac{1}{p-1}\inf\{\omega(J(S_i/R_i))\}\in \frac{1}{p-1}\omega(L).
$$

In contrast, we can get any nonnegative  real number as the distance ${\rm dist}(\omega/\nu)$ on  $K$ if we allow a sequence which does not stabilize to type 2, as is shown in the following example.

\begin{Theorem}\label{Example1} 
Suppose that $K$ is an algebraic function field of transcendence degree 2 over an algebraically closed field $k$ of characteristic $p>0$, and that $A$ is an algebraic  regular local ring of $K$ with regular parameters $z$ and $w$. Let $\alpha\in \RR_{\ge 0}$ and let $\Phi:\NN\rightarrow \{1,2\}$ be a function such that $\Phi(n)$ is not identically equal to 2 for $n\gg 0$ . Then there exists an Artin-Schreier extension $K\rightarrow L$ and a sequence (\ref{eq30}) such that $R_0=A$,
$R_r\rightarrow S_r$ is of type 1 if $\Phi(r)=1$ and of type 2 if $\Phi(r)=2$
and the induced defect extension of valuations satisfies
$$
{\rm dist}(\omega/\nu)=-\alpha,
$$
 where the valuation $\nu$ of $L$ is normalized so that $\nu(z)=1$.  We will further have that $m_i>1$ and $\overline m_i>1$ for all $i$ in the sequence (\ref{eq30}).
\end{Theorem}

\begin{proof} 
First assume that $\Phi(0)=1$.
%We can assume that $\Phi(0)=1$. Otherwise, we reduce to this case by setting 
%$$
%\Phi'(n)=\left\{\begin{array}{ll}
%1&\mbox{ if }n=0\\
%\Phi(n-1)&\mbox{ if }n\ge 1.
%\end{array}\right.
%$$
%and solving the problem for $\Phi'$, which immediately gives a solution for the function $\Phi$.
%Let $R_0$ be a regular local ring of $K$ with regular parameters $u,w$. 
Let $R_0=A$, $u_0=z$ and $v_0=w$.
Let $e$ be a positive integer  such that $e>\alpha$. Let $c_0=(p-1)e$. Let $\Theta$ be a root of the Artin-Schreier polynomial $X^p-X-v_0u_0^{-pe}$. Let $L=K(\Theta)$. Set $x_0=u_0$, $y_0=u_0^e\Theta$. Let $S_0=R[y_0]_{(x_0,y_0)}$, which is an algebraic regular local ring of $L$ which dominates $R_0$. The regular parameters $x_0,y_0$ in $S_0$ satisfy $u_0=x_0, v_0=y_0^p-x_0^{e(p-1)}y_0$, so that the extension $R\rightarrow S$ is of type 1. We have that $J(S_0/R_0)=(x_0^{c_0})$, with
$\frac{c_0}{p-1}>\alpha$.

Suppose that we have a sequence 
$$
\begin{array}{ccc}
S_r&\rightarrow &S_{r+1}\\
\uparrow&&\uparrow\\
R_r&\rightarrow & R_{r+1}
\end{array}
$$
where $R_r\rightarrow S_{r}$ and $R_{r+1}\rightarrow S_{r+1}$ are both of type 1. Then
from Theorem \ref{TheoremA}, we have that 
\begin{equation}\label{eq60}
\left(\frac{c_{r+1}}{p-1}\right)\frac{1}{m_1\cdots m_{r+1}}=\left(\frac{c_r}{p-1}\right)\frac{1}{m_1\cdots m_{r}}
-\frac{q_{r+1}}{m_{r+1}}\left(\frac{1}{m_1\cdots m_{r}}\right).
\end{equation}

Suppose that we have a sequence
$$
\begin{array}{ccccccccc}
S_r&\rightarrow &S_{r+1}&\rightarrow &\cdots&\rightarrow &S_{r+s}&\rightarrow &S_{r+s+1}\\
\uparrow &&\uparrow&&&&\uparrow&&\uparrow\\
R_r&\rightarrow &R_{r+1}&\rightarrow &\cdots&\rightarrow &R_{r+s}&\rightarrow &R_{r+s+1}\end{array}
$$
where $s\ge 1$, $R_i\rightarrow S_i$ is of type 1 if $i=r$ or $i=r+s+1$ and $R_i\rightarrow S_i$ is of type 2 if $r+1\le i\le r+s$. Then from Theorems \ref{TheoremA} and \ref{TheoremB}, we have that
\begin{equation}\label{eq61}
\left(\frac{c_{r+s+1}}{p-1}\right)\frac{1}{m_1\cdots m_{r+s+1}}
=\left(\frac{c_r}{p-1}\right)\frac{1}{m_1\cdots m_{r}}-\frac{q_{r+1}}{m_{r+1}}\left(\frac{1}{m_1\cdots m_{r}}\right).
\end{equation}

Let $p'$ be a prime distinct from $p$.

We now inductively construct the sequence (\ref{eq30}), so that $m_i>1$ and $\overline m_i>1$ for all $i$. Suppose that the sequence has been constructed up to $R_r\rightarrow S_r$, $\Phi(r)=1$  and  we have that  for all $t\le r$ such that $\Phi(t)=1$,
\begin{equation}\label{eq66}
\alpha<\left(\frac{c_t}{p-1}\right)\frac{1}{m_1\cdots m_{t}}\mbox{ and }
\left(\frac{c_t}{p-1}\right)\frac{1}{m_1\cdots m_{t}}<\alpha+\frac{1}{2^{t}}\mbox{ if }t>0.
\end{equation}
  
First suppose that  $\Phi(r+1)=1$. 
There exists $\lambda(r+1)\in \ZZ_+$ such that there exits $q_{r+1}\in \ZZ_+$ such that $\mbox{gcd}(q_{r+1},p')=1$ and
\begin{equation}\label{eq62}
\frac{c_r}{p-1}-\alpha m_1\cdots m_{r}
>\frac{q_{r+1}}{(p')^{\lambda(r+1)}}>\frac{c_r}{p-1}-\left(\alpha+\frac{1}{2^{r+1}}\right)m_1\cdots m_{r}.
\end{equation}
Set $m_{r+1}=(p')^{\lambda(r+1)}$. Then
\begin{equation}\label{eq63}
\alpha+\frac{1}{2^{r+1}}>\left(\frac{c_r}{p-1}\right)\frac{1}{m_1\cdots m_{r}} -\left(\frac{q_{r+1}}{m_{r+1}}\right)\frac{1}{m_1\cdots m_{r}}>\alpha.
\end{equation}

Now we have that $\frac{q_{r+1}}{m_{r+1}}<\frac{c_r}{p-1}$ with $\mbox{gcd}(m_{r+1},pq_{r+1})=1$ so we may define from Theorem \ref{TheoremA} and the above values of $q_{r+1}$ and $m_{r+1}$ a commutative diagram
$$
\begin{array}{ccc}
S_r&\rightarrow &S_{r+1}\\
\uparrow&&\uparrow\\
R_r&\rightarrow & R_{r+1}
\end{array}
$$
such that $R_{r+1}\rightarrow S_{r+1}$ is of type 1. Further, (\ref{eq66}) holds for $t=r+1$ by (\ref{eq60}) and (\ref{eq63}).

Now suppose that  $\Phi(r+1)=2$. Let $s\ge 1$ be the smallest integer such that $\Phi(r+s+1)=1$. 
There exists an integer $\lambda(r+1)>1$ such that there exits $q_{r+1}\in \ZZ_+$ such that $\mbox{gcd}(q_{r+1},p)=1$ and
\begin{equation}\label{eq64}
\frac{c_r}{p-1}-\alpha m_1\cdots m_{r}
>\frac{q_{r+1}}{p^{\lambda(r+1)}}>\frac{c_r}{p-1}-\left(\alpha+\frac{1}{2^{r+1}}\right)m_1\cdots m_{r}.
\end{equation}
Set $m_{r+1}=p^{\lambda(r+1)}$. Then
\begin{equation}\label{eq65}
\alpha+\frac{1}{2^{r+1}}>\left(\frac{c_r}{p-1}\right)\frac{1}{m_1\cdots m_{r}} -\left(\frac{q_{r+1}}{m_{r+1}}\right)\frac{1}{m_1\cdots m_{r}}>\alpha.
\end{equation}

We have that $\frac{q_{r+1}}{m_{r+1}}<\frac{c_r}{p-1}$ with $\mbox{gcd}(m_{r+1},pq_{r+1})=p$ so we may define from Theorem \ref{TheoremA} and the above values of $q_{r+1}$ and $m_{r+1}$ a commutative diagram
$$
\begin{array}{ccc}
S_r&\rightarrow &S_{r+1}\\
\uparrow&&\uparrow\\
R_r&\rightarrow & R_{r+1}
\end{array}
$$
such that $R_{r+1}\rightarrow S_{r+1}$ is of type 2. We have $\sigma=p$ in Theorem \ref{TheoremA} and $\overline m_{r+1}=\frac{m_{r+1}}{\sigma}>1$.
For $r+1\le i\le r+s$ define
$$
\begin{array}{ccc}
S_i&\rightarrow &S_{i+1}\\
\uparrow&&\uparrow\\
R_i&\rightarrow & R_{i+1}
\end{array}
$$
from Theorem \ref{TheoremB} by taking $m_{i+1}=(p')^2$ and $q_{i+1}=p^2$ if $i<s$ and taking $m_{i+1}=p^2$ and $q_{i+1}=(p')^2$ if $i=s$. From Theorem \ref{TheoremB} we have a commutative diagram
$$
\begin{array}{ccccccccc}
S_r&\rightarrow &S_{r+1}&\rightarrow &\cdots&\rightarrow &S_{r+s}&\rightarrow &S_{r+s+1}\\
\uparrow &&\uparrow&&&&\uparrow&&\uparrow\\
R_r&\rightarrow &R_{r+1}&\rightarrow &\cdots&\rightarrow &R_{r+s}&\rightarrow &R_{r+s+1}\end{array}
$$
such that $R_i\rightarrow S_i$ is of type 1 if $i=r$ or $i=r+s+1$ and $R_i\rightarrow S_i$ is of type 2 if $r+1\le i\le r+s$. Further, (\ref{eq66}) is satisfied with $t=r+s+1$ by (\ref{eq61}) and (\ref{eq65}). We saw above that $m_{r+1}>1$ and $\overline m_{r+1}>1$.
If $r+1<i\le r+s$, we have $\sigma=p$ in Theorem \ref{TheoremB} so that $\overline m_i=\frac{pm_i}{\sigma}=m_i>1$. If $i=r+s+1$, then $\sigma=1$ in Theorem \ref{TheoremB} and $\overline m_{r+s+1}=pm_{r+s+1}>1$. Thus $m_i>1$ and $\overline m_i>1$ for $r\le i\le r+s+1$.

Now by Proposition \ref{Prop100}, we have that
$$
-{\rm dist}(\omega/\nu)=\frac{1}{p-1}\inf_i\{\omega(J(S_i/R_i))\}=\alpha\omega(x_0)=\alpha.
$$

Now suppose that $\Phi(0)=2$. Using the construction of the above case (when $\Phi(0)=1$), we can in this case construct an augmented sequence 
\begin{equation}\label{eq80}
\begin{array}{ccccccccc}
B&\rightarrow &S_{-1}&\rightarrow &S_0&\rightarrow & S_1&\rightarrow & \cdots\\
\uparrow&&\uparrow&&\uparrow&&\uparrow\\
A&\rightarrow &R_{-1}&\rightarrow &R_0&\rightarrow & R_1&\rightarrow & \cdots
\end{array}
\end{equation}
where $\Phi$ is extended to the set $\{-1,0,1,2,\ldots\}$ by defining $\Phi(-1)=1$, and such that the conclusions of the theorem hold for this augmented sequence. We then get the statement of the theorem by forgetting the map $R_{-1}\rightarrow S_{-1}$.
\end{proof}

\begin{Remark}\label{RemarkEX} In the construction of the sequence (\ref{eq30}) in Theorem \ref{Example1}, we have $m_i>1$ and $\overline m_i>1$ for all $i$, so that $\omega( \overline y_i)\not\in \omega(\overline x_i)\ZZ$ for all $i$ and $\omega(\overline v_i)\not\in \omega(u_i)\ZZ$ for all $i$. Thus $R_0\rightarrow R_1\rightarrow R_2\rightarrow \cdots$ is the sequence of sequences of standard quadratic transforms along $\nu$ and $S_0\rightarrow S_1\rightarrow S_2\rightarrow \cdots$ is the sequence of sequences of standard quadratic transforms along $\omega$.
\end{Remark}

\section{An example of a tower of independent defect extensions in which strong local monomialization doesn't hold}

\begin{Theorem}\label{Example3} There exists a tower $(K,\nu)\rightarrow (L,\omega)\rightarrow (M,\mu)$ of independent defect Artin-Schreier extensions of valued two dimensional algebraic function fields over an algebraically closed field $k$ of  characteristic $p>0$ such that there exist algebraic regular local rings $A$ of $K$ and $C$ of $M$ such that $\mu$ dominates $C$ and $C$ dominates $A$ but strong local monomialization along $\mu$ does not hold above $A\rightarrow C$. 
\end{Theorem} 

\begin{Remark}\label{RemarkA'} Let $\delta\in \RR_{\ge 0}$ be a fixed ratio. Suppose that $R\rightarrow S$ is of type 1. By taking $m$ and $q$ sufficiently large in Theorem \ref{TheoremA} such that   $R_1\rightarrow S_1$ is of type 2, we can achieve that $v_1=\lambda y_1+g(x_1)$ where $\lambda$ is a unit in $S_1$ and the order of $g(x_1)$ is arbitrarily large.  Suppose that $R\rightarrow S$ is of type 2. By taking $m$ and $q$ sufficiently large in Theorem \ref{TheoremB}  such that  $R_1\rightarrow S_1$ is of type 1 we can achieve that $v_1=y_1^p\gamma+x_1^{c_1}y_1\tau+f(x_1)$ where $\gamma$ and $\tau$ are unit series in $S_1$ and the order of $f(x_1)$ is arbitrarily large. In both cases, we can choose $m$ and $q$ so that $\frac{q}{m}$ is arbitrarily close to $\delta$.
\end{Remark}

\begin{Remark}\label{RemarkB'}
In Theorem \ref{TheoremB}, we have an expression $\overline v=\tau y+f(x)$ where $\tau$ is a unit in $S$. Suppose that $m$ and $q$ are positive integers with $\mbox{gcd}(m,q)=1$ and such that $\mbox{ord }f(x)>\frac{q}{m}$. Then the proof of Theorem \ref{TheoremB} extends to show that the conclusions of Theorem \ref{TheoremB} hold with $\overline y$ replaced with $y$.
\end{Remark}

We now give the proof of Theorem \ref{Example3}.
\begin{proof} 
Let $K$ be a two dimensional algebraic function field over an algebraically closed field, and let $R_{-2}$ be a two dimensional algebraic regular local ring of $K$. Let $u_{-2},v_{-2}$ be regular parameters in $R_{-2}$.

%$R_0=A$, $u_0=z$ and $v_0=w$.

Let $e$ be a positive integer. Let $c_{-2}=(p-1)e$. Let $\Theta$ be a root of the Artin-Schreier polynomial $X^p-X-v_{-2}u_{-2}^{-pe}$. Let $L=K(\Theta)$. Set $x_{-2}=u_{-2}$, $y_{-2}=u_{-2}^e\Theta$. Let $S_{-2}=R_{-2}[y_{-2}]_{(x_{-2},y_{-2})}$, which is an algebraic regular local ring of $L$ which dominates $R_{-2}$. The regular parameters $x_{-2},y_{-2}$ in $S_{-2}$ satisfy $u_{-2}=x_{-2}, v_{-2}=y_{-2}^p-x_{-2}^{e(p-1)}y_{-2}$, so that the extension $R_{-2}\rightarrow S_{-2}$ is of type 1. We have that 
$J(S_{-2}/R_{-2})=(x_{-2}^{c_{-2}})$, with
$\frac{c_{-2}}{p-1}>0$.

We first construct a commutative diagram
$$
\begin{array}{lll}
S_{-2}&\rightarrow& S_{-1}\\
\uparrow&&\uparrow\\
R_{-2}&\rightarrow &R_{-1}
\end{array}
$$
using Theorem \ref{TheoremA} so that $R_{-1}\rightarrow S_{-1}$ is of type 2.
Let $\Sigma$ be a root of the Artin-Schreier polynomial $X^p-X-y_{-1}x_{-1}^{-pe}$. Let $M=L(\Sigma)$. Set $z_{-1}=x_{-1}$, $w_{-1}=x_{-1}^e\Sigma$. Let $T_{-1}=S_{-1}[w_{-1}]_{(z_{-1},w_{-1})}$, which is an algebraic regular local ring of $M$ which dominates $S_{-1}$. The regular parameters $z_{-1},w_{-1}$ in $T_{-1}$ satisfy $x_{-1}=z_{-1}, y_{-1}=w_{-1}^p-z_{-1}^{e(p-1)}w_{-1}$, so that the extension $S_{-1}\rightarrow T_{-1}$ is of type 1. We have that $J(T_{-1}/S_{-1})=(z_{-1}^{c_{-1}'})$, with
$\frac{c_{-1}'}{p-1}>0$.

From Theorems \ref{TheoremA} and \ref{TheoremB}, we construct
$$
\begin{array}{ccc}
T_{-1}&\rightarrow &T_0\\
\uparrow&&\uparrow\\
S_{-1}&\rightarrow &S_0\\
\uparrow&&\uparrow\\
R_{-1}&\rightarrow&R_0\\
\end{array}
$$
such that $R_0\rightarrow S_0$ is of type 1 and $S_0\rightarrow T_0$ is of type 2. Explicitely, $R_{-1},R_0, S_{-1}, S_0, T_{-1}, T_0$ have respective regular parameters $(u_{-1},v_{-1})$, $(u_0,v_0)$, $(x_{-1},y_{-1})$, $(x_0,y_0)$ and $(z_{-1}, w_{-1})$, $(z_0,w_0)$ which are related by equations
$$
\begin{array}{l}
u_{-1}=u_0^{pm_0}(v_0+\beta_0)^{d_0'}, v_{-1}=u_0^{q_0}(v_0+\beta_0)^{e_0'}\\
x_{-1}=x_0^{m_0}(y_0+\alpha_0)^{a_0'}, y_{-1}=x_0^{q_0}(w_0+\alpha_0)^{g_0'}\\
z_{-1}=z_0^{pm_0}(v_0+\gamma_0)^{f_0'}, w_{-1}=z_0^{q_0}(w_0+\gamma_0)^{g_0'}
\end{array}
$$
where $p\not| q_0$ and $\frac{q_0}{pm_0}<\frac{c_{-1}'}{p-1}$ where $J(T_{-1}/S_{-1})=(z_{-1}^{c_{-1}'})$.

By Remarks \ref{RemarkA'} and \ref{RemarkB'}, we can construct $R_0\rightarrow S_0\rightarrow T_0$ so that we have expressions $y_0=\lambda_0w_0+g_0(z_0)$ where $\lambda_0$ is a unit in $T_0$ and $\mbox{ord } g_0(z_0)$ is arbitrarily large and $v_0=\sigma_0y_0^p+\tau_0x_0^{c_0}y_0+f_0(x_0)$ where $\sigma_0,\tau_0$ are units in $S_0$ and $\mbox{ord } f_0(x_0)$ is arbitrarily large.

We will inductively construct a commutative diagram within $K\rightarrow L\rightarrow M$ of two dimensional regular algebraic local rings
\begin{equation}\label{eq5**}
\begin{array}{ccccccc}
T_0&\rightarrow &T_1&\rightarrow& T_2&\rightarrow &\cdots\\
\uparrow&&\uparrow&&\uparrow&&\\
S_0&\rightarrow&S_1&\rightarrow &S_2&\rightarrow &\cdots\\
\uparrow&&\uparrow&&\uparrow&&\\
R_0&\rightarrow&R_1&\rightarrow &R_2&\rightarrow&\cdots\\
\end{array}
\end{equation}
such that $R_i\rightarrow S_i$ is of type 1 if $i$ is even and is of type 2 if $i$ is odd, $S_i\rightarrow T_i$ is of type 2 if $i$ is even and is of type 1 if $i$ is odd. Further, valuations $\nu$, $\omega$ and $\mu$ of the respective function fields $K$, $L$ and $M$ determined by these sequences are such that $K\rightarrow L$ and $L\rightarrow M$ are independent defect extensions. 
We will have that $R_i$ has regular parameters $(u_i,v_i)$, $S_i$ has regular parameters $(x_i,y_i)$ and $T_i$ has regular parameters $(z_i,w_i)$ such that
$$
u_i=u_{i+1}^{\overline m_{i+1}}(v_{i+1}+\beta_{i+1})^{d_{i+1}'}, v_i=u_{i+1}^{\overline q_{i+1}}(v_{i+1}+\beta_{i+1})^{e_{i+1}'},
$$
$$
x_i=x_{i+1}^{m_{i+1}}(y_{i+1}+\alpha_{i+1})^{a_{i+1}'}, y_i=x_{i+1}^{q_{i+1}}(y_{i+1}+\alpha_{i+1})^{b_{i+1}'},
$$
$$
z_i=z_{i+1}^{m_{i+1}'}(w_{i+1}+\gamma_{i+1})^{f_{i+1}'}, w_i=z_{i+1}^{q_{i+1}'}(w_i+\gamma_{i+1})^{g_{i+1}'}
$$
with $\overline m_i, m_i$ and $m_i'$ larger than 1 for all $i$.

Let $J(S_i/R_i)=(x_i^{c_i})$ and $J(T_i/S_i)=(z_i^{c_i'})$.

If $i$ is even, then 
$m_{i+1}=p\overline m_{i+1}, m_{i+1}'=\overline m_{i+1},q_{i+1}=\overline q_{i+1}, q_{i+1}'=q_{i+1}$ and 
$$
\frac{q_{i+1}}{m_{i+1}}<\frac{c_i}{p-1}.
$$

If $i$ is odd, then $\overline m_{i+1}=pm_{i+1}, m_{i+1}'=\overline m_{i+1}, q_{i+1}=\overline q_{i+1}, q_{i+1}'=q_{i+1}$ and 
$$
\frac{q_{i+1}'}{m_{i+1}'}<\frac{c_i'}{p-1}.
$$

In our construction, if $r$ is even, we will have that
\begin{equation}\label{eq100*}
y_r=\lambda_rw_r+g_r(z_r)
\end{equation}
where $\lambda_r$ is a unit in $T_r$ and $\mbox{ord } g_r(z_r)$ is arbitrarily large and
\begin{equation}\label{eq101}
v_r=\sigma_ry_r^p+\tau_rx_r^{c_r}y_r+f_r(x_r)
\end{equation}
where $\sigma_r,\tau_r$ are units in $S_r$ and $\mbox{ord }f_r(x_r)$ is arbitrarily large. 
If $r$ is even, we will have
\begin{equation}\label{eq106}
y_r=\sigma_rw_r^p+\tau_rz_r^{c_r'}w_r+f(z_r)
\end{equation}
where $\sigma_r,\tau_r$ are units in $T_r$ and $\mbox{ord }f(z_r)$ is arbitrarily large and
\begin{equation}\label{eq107}
v_r=\lambda_ry_r+g_r(x_r)
\end{equation}
where $\lambda_r$ is a unit in $S_r$ and $\mbox{ord }g_r(x_r)$ is arbitrarily large.

Suppose that $r$ is even, and we have constructed $R_r\rightarrow S_r\rightarrow T_r$. We will construct
$$
\begin{array}{ccccc}
T_r&\rightarrow&T_{r+1}&\rightarrow&T_{r+2}\\
\uparrow&&\uparrow&&\uparrow\\
S_r&\rightarrow&S_{r+1}&\rightarrow&S_{r+2}\\
\uparrow&&\uparrow&&\uparrow\\
R_r&\rightarrow&R_{r+1}&\rightarrow &R_{r+2}.
\end{array}
$$
There exists an integer $\lambda(r+1)>1$ and $q_{r+1}\in \ZZ_+$ such that $\mbox{gcd}(q_{r+1},p)=1$ and
\begin{equation}\label{eq64*}
\frac{c_r}{p-1}
>\frac{q_{r+1}}{p^{\lambda(r+1)}}>\frac{c_r}{p-1}-\frac{1}{2^{r+1}}m_1\cdots m_{r}.
\end{equation}
In fact, we can find $\lambda(r+1)$ arbitrarily large satisfying the inequality.
Set $m_{r+1}=p^{\lambda(r+1)}$. 
We have that $\frac{q_{r+1}}{m_{r+1}}<\frac{c_r}{p-1}$ with $\mbox{gcd}(m_{r+1},pq_{r+1})=p$. 
This choice of $m_{r+1}$ and $q_{r+1}$ (along with a choice of $0\ne\alpha_{r+1}\in k$) determines $S_r\rightarrow S_{r+1}$. We have an expression
$v_r=\sigma_r y_r^p+\tau_r x_r^{c_r}y_r+f_r(x_r)$ where $\mbox{ord }f_r(x_r)$ is arbitrarily large. In particular, we can assume that
$\mbox{ord }f_r(x_r)>\frac{pq_{r+1}}{m_{r+1}}$. Then $R_r\rightarrow R_{r+1}$ is defined as desired by Theorem \ref{TheoremA}.
By Remark \ref{RemarkA'}, since we can take $\lambda(r+1)$ to be arbitrarily large, we can assume that $v_{r+1}=\lambda_{r+1}y_{r+1}+g_{r+1}(x_{r+1})$ where $\mbox{ord }g_{r+1}(x_{r+1})$ is arbitrarily large. 

By Remark \ref{RemarkB'} and Theorem \ref{TheoremB}, $T_r\rightarrow T_{r+1}$ is defined as desired, with $m_{r+1}'=\frac{m_{r+1}}{p}$, $q_{r+1}'=q_{r+1}$. Since we can take $\lambda(r+1)$ to be arbitrarily large, we can assume that
$y_{r+1}=\sigma_{r+1}w_{r+1}^p+\tau_{r+1}z_{r+1}^{c_{r+1}'}w_r+f_{r+1}(z_{r+1})$ where $\mbox{ord }f_{r+1}(z_{r+1})$ is arbitrarily large. 

We have defined a commutative diagram
%$$
%\begin{array}{ccc}
%S_r&\rightarrow &S_{r+1}\\
%\uparrow&&\uparrow\\
%R_r&\rightarrow & R_{r+1}
%\end{array}
%$$
%such that $R_{r+1}\rightarrow S_{r+1}$ is of type 2. We have $\sigma=p$ in Theorem \ref{TheoremA} and $\overline m_{r+1}=\frac{m_{r+1}}{\sigma}>1$. by Theorem \ref{TheoremB}, we may construct $T_r\rightarrow T_{r+1}$ giving a commutative diagram
\begin{equation}\label{eq103}
\begin{array}{ccc}
T_r&\rightarrow &T_{r+1}\\
\uparrow&&\uparrow\\
S_r&\rightarrow &S_{r+1}\\
\uparrow&&\uparrow\\
R_r&\rightarrow & R_{r+1}
\end{array}
\end{equation}
with the desired properties; in particular, $R_{r+1}\rightarrow S_{r+1}$ is of type 2 with
$$
\frac{c_{r+1}}{p-1}=\left(\frac{c_r}{p-1}\right)m_{r+1}-q_{r+1}+1
$$
and $S_{r+1}\rightarrow T_{r+1}$ is of type 1, with 
$$
\frac{c_{r+1}'}{p-1}=\frac{c_r'}{p-1}m_{r+1}'-m_{r+1}'.
$$

%By Remark \ref{RemarkA'}, we have expressions 
%$$
%y_{r+1}=\gamma_{r+1}w_{r+1}^p+\tau_{r+1}z_{r+1}^{c_{r+1}'}w_{r+1}+f_{r+1}(z_{r+1})
%$$
%where $\gamma_{r+1}, \tau_{r+1}$ are units in $T_{r+1}$ and $\mbox{ord }f_{r+1}(z_{r+1})$ is arbitrarily large. Further, 
%$$
%v_{r+1}=\lambda_{r+1}y_{r+1}+g_{r+1}(x_{r+1})
%$$
%where $\lambda_{r+1}$ is a unit in $S_{r+1}$ and $\mbox{ord }g_{r+1}(x_{r+1})$ is arbitrarily large.

Now choose $q_{r+2}'$, $m_{r+2}'=p^{\lambda(r+2)}$ such that $p\not|q_{r+2}'$ and
\begin{equation}\label{eq1**}
\frac{c_{r+1}'}{p-1}>\frac{q_{r+2}'}{m_{r+2}'}>\frac{c_{r+1}'}{p-1}-\frac{1}{2^{r+2}}m_1'\cdots m_{r+1}'.
\end{equation} 
We can take $\lambda(r+2)$ arbitrarily large.
Set $m_{r+2}=\frac{m_{r+2}'}{p}=p^{\lambda(r+2)-1}$, $q_{r+2}=q_{r+2}'$. By (\ref{eq1**}),
$\frac{q_{r+2}'}{m_{r+2}'}<\frac{c_{r+1}'}{p-1}$.

%%%%%%%%%%%%%
%Now choose $q_{r+2}$ and $m_{r+2}=p^{\lambda(r+2)}$ such that $p\not| q_{r+2}$ and
%$$
%\frac{c_r'}{p-1}-1-\frac{m_1\cdots m_r}{2^r}<\frac{q_{r+2}-1}{m_{r+1}m_{r+2}}<\frac{c_r'}{p-1}-1.
%$$
%Then 
%\begin{equation}\label{eq1**}
%\frac{c_r'}{p-1}-1+\frac{1}{m_{r+1}m_{r+2}}-\frac{m_1\cdots m_r}{2^{r+1}}<\frac{q_{r+2}}{m_{r+1}m_{r+2}}
%\end{equation}
%and
%\begin{equation}\label{eq2**}
%\frac{q_{r+2}}{m_{r+1}m_{r+2}}<\frac{c_r'}{p-1}-1+\frac{1}{m_{r+1}m_{r+2}}.
%\end{equation}
Now construct, as in the construction of (\ref{eq103}), using Theorems \ref{TheoremA} and \ref{TheoremB} and Remark \ref{RemarkB'} and these values of $m_{r+2}$ and $q_{r+2}$,
$$
\begin{array}{ccc}
T_{r+1}&\rightarrow &T_{r+2}\\
\uparrow&&\uparrow\\
S_{r+1}&\rightarrow &S_{r+2}\\
\uparrow&&\uparrow\\
R_{r+1}&\rightarrow & R_{r+2},
\end{array}
$$
so that $R_{r+2}\rightarrow S_{r+2}$ is of type 1 and $S_{r+2}\rightarrow T_{r+2}$ is of type 2. By Remark \ref{RemarkA'}, we obtain expressions (\ref{eq100*}) and (\ref{eq101}) for $r+2$.

 By induction, we construct the diagram (\ref{eq5**}).

  Let $A=R_0$ and $C=T_0$. We will show that strong local monomialization doesn't hold above $A\rightarrow C$ along $\mu$. Suppose that $R'\rightarrow T'$ has a strongly monomial form above $A\rightarrow C$. Then $R'$ has regular parameters $u',v'$ and $T'$ has regular parameters $z',w'$ such that $u'=\lambda (z')^m$ and $v'=w'$ where $m\in \ZZ_{>0}$ and $\lambda$ is a unit in $T'$. We will show that this cannot occur. There exists a commutative diagram
  $$
  \begin{array}{ccccc}
  T_s&\rightarrow&T'&\rightarrow&T_{s+1}\\
  \uparrow&&\uparrow&&\uparrow\\
  R_s&\rightarrow&R'&\rightarrow&R_{s+1}
  \end{array}
  $$
  for some $s$. The ring $T'$ has regular parameters $\overline z,\overline w$ such that 
  \begin{equation}\label{eq108}
  z_s=\overline z^a\overline w^b, w_s=\overline z^c\overline w^d
  \end{equation}
   for some $a,b,c,d\in \NN$ with $ad-bc=\pm1$, and $R'$ has regular parameters $\overline u,\overline v$ such that $u_s=\overline u^{\overline a}\overline v^{\overline b}$, $v_s=\overline u^{\overline c}\overline v^{\overline d}$, where $\overline a \overline d-\overline b \overline c=\pm 1$. We have an expression 
  \begin{equation}\label{eq102}
  u_s=\alpha z_s^p, v_s=\beta w_s^p+\Omega
  \end{equation}
  where $\alpha,\beta$ are units in $T_s$ and where 
  \begin{equation}\label{eq109}
  \Omega=\epsilon z_s^{pc_s}w_s+M
  \end{equation}
  or
  \begin{equation}\label{eq110}
  \Omega=\epsilon z_s^{c_s'}w_s+M
  \end{equation}
  where $\epsilon\in T_s$ is a unit and $M$ is a sum of monomials in $z_s,w_s$ of high order in $z_s$. Further,
  $\mu(w_s^p)<\mu(z_s^{pc_s}w_s)$ in (\ref{eq109}) and
  $\mu(w_s^p)<\mu(z_s^{c_s'}w_s)$ in (\ref{eq110}).
  
  In particular, $R_s\rightarrow T_s$ is not a strongly monomial form. 
  
   Substituting (\ref{eq108}) into $u_s$ and $v_s$ in (\ref{eq102}), we have 
  \begin{equation}\label{eq113}
  u_s=\alpha \overline z^{ap}\overline w^{bp}, v_s=\beta \overline z^{cp}\overline w^{dp}+\Omega.
  \end{equation}
  We necessarily have that $u_s|v_s$ or $v_s|u_s$ in $T'$. 
  
  First suppose that $c\ge a$ and $d\ge b$. Then we have that
  $$
  u_s=\alpha \overline z^{ap}\overline w^{bp},
  \frac{v_s}{u_s}=\beta\overline z^{(c-a)p}\overline w^{(d-b)p}+\frac{\Omega}{\alpha \overline z^{ap}\overline w^{bp}}
  $$
  giving an expression of the form (\ref{eq113}).
  We will show that this is not a strongly monomial form. If it is, then we must have that $a=0$ or $b=0$ so that either 
  \begin{equation}\label{eq111}
  z_s=\overline w, w_s=\overline z \overline w^d
  \end{equation}
  or
  \begin{equation}\label{eq112}
  z_s=\overline z, w_s=\overline z^c \overline w
  \end{equation}
  and we must have that $\frac{\Omega}{u_s}$ is part of a regular system of parameters in $T'$. Substituting into (\ref{eq109}) or (\ref{eq110}), we see that this  cannot occur except possibly in the case that (\ref{eq110}) holds and $\frac{z_s^{c_s'}w_s}{u_s}$ is part of a regular system of parameters in $T'$.

 Suppose that  (\ref{eq110}) and (\ref{eq11}) hold with
  $$
  \frac{z_s^{c_s'}w_s}{u_s}=\frac{\overline w^{c_s'+d}\overline z}{\alpha \overline w^p}
  $$
  being part of a regular system of parameters in $T'$. Now in this case, $\mu(w_s)>\mu(z_s)$ and $\mu(w_s^p)<\mu(z_s^{c_s'}w_s)$ so $p\le c_s'$. Thus $\frac{\overline w^{c_s'+d}\overline z}{\alpha \overline w^p}$ cannot be part of a regular system of parameters in $T'$.   A similar argument shows that we do not obtain a strongly monomial form when (\ref{eq110}) and (\ref{eq112}) hold. 
  
  Suppose that $c<a$ and $d<b$. Then we have expressions
  $$
  v_s=\gamma\overline z^{cp}\overline w^{dp},
  \frac{u_s}{v_s}=\alpha \gamma^{-1}\overline z^{(a-c)p}\overline w^{(b-d)p}
  $$
  where $\gamma\in T'$ is a unit, giving an expression of the form of (\ref{eq113}), which is not strongly monomial. Thus we reduce to the case where $(c-a)(d-b)<0$. We then have that $u_s\not | v_s$ since $u_s\not |\overline z^{cp}\overline w^{dp}$. Suppose that $v_s|u_s$. Then $v_s=\lambda \overline z^{cp}\overline w^{dp}$ where $\lambda$ is a unit in $T'$. But this is impossible since $(c-a)(d-b)<0$.
 Thus $R'\rightarrow T'$ has a form (\ref{eq113}) with $a,b,c,d>0$ and so 
 cannot be a strongly monomial form. We have established that strong local monomialization along $\mu$ does not hold above $A\rightarrow C$.

From Theorem \ref{TheoremA}, we have that 
\begin{equation}\label{eq60*}
\left(\frac{c_{r+1}}{p-1}\right)\frac{1}{m_1\cdots m_{r+1}}=\left(\frac{c_r}{p-1}\right)\frac{1}{m_1\cdots m_{r}}
-\frac{q_{r+1}}{m_{r+1}}\left(\frac{1}{m_1\cdots m_{r}}\right)+\frac{1}{m_1\cdots m_{r+1}}.
\end{equation}
Then from Theorem \ref{TheoremB}, we have that
$$
\frac{c_{r+2}}{p-1}=\left(\frac{c_{r+1}}{p-1}\right)m_{r+2}-m_{r+2},
$$
and so
\begin{equation}\label{eq61*}
\left(\frac{c_{r+2}}{p-1}\right)\frac{1}{m_1\cdots m_{r+2}}
=\left(\frac{c_r}{p-1}\right)\frac{1}{m_1\cdots m_{r}}-\frac{q_{r+1}}{m_1\cdots m_{r+1}}.
\end{equation}
By equation (\ref{eq64*})  we have
\begin{equation}\label{eq65*}
\frac{1}{2^{r+1}}>\left(\frac{c_r}{p-1}\right)\frac{1}{m_1\cdots m_{r}} -\left(\frac{q_{r+1}}{m_{r+1}}\right)\frac{1}{m_1\cdots m_{r}}>0.
\end{equation}

%By Theorem \ref{TheoremA}, we have that
%$$
%\frac{c_{r+2}'}{p-1}=\frac{c_{r+1}'}{p-1}m_{r+2}-q_{r+2}+1
%=\frac{c_r'}{p-1}m_{r+1}m_{r+2}-m_{r+1}m_{r+2}-q_{r+2}+1.
%$$
%Thus
%\begin{equation}\label{eq4**}
%\frac{c_{r+2}'}{p-1}\frac{1}{m_1\cdots m_{r+2}}=\frac{c_r'}{p-1}\frac{1}{m_1\cdots m_r}-\frac{1}{m_1\cdots m_r}
%-\frac{q_{r+2}}{m_1\cdots m_{r+2}}+\frac{1}{m_1\cdots m_{r+2}}.
%\end{equation}
%By equations (\ref{eq1**}) and (\ref{eq2**}), we have that 
%\begin{equation}\label{eq3**}
%0<\frac{c_r'}{p-1}\frac{1}{m_1\cdots m_r}-\frac{1}{m_1\cdots m_r}
%-\frac{q_{r+2}}{m_1\cdots m_{r+2}}+\frac{1}{m_1\cdots m_{r+2}}<\frac{1}{2^{r+1}}.
%\end{equation}

By Theorem \ref{TheoremA},
$$
\left(\frac{c_{r+2}'}{p-1}\right)\frac{1}{m_1'\cdots m_{r+2}'}=\left(\frac{c_{r+1}'}{p-1}\right)\frac{1}{m_1'\cdots m_{r+1}'}
-\frac{q_{r+2}'}{m_1'\cdots m_{r+2}'}+\frac{1}{m_1'\cdots m_{r+2}'}
$$
and by Theorem \ref{TheoremB},
$$
\frac{c_{r+3}'}{p-1}=\left(\frac{c_{r+2}'}{p-1}\right)m_{r+3}'-m_{r+3}'.
$$
We thus have that
\begin{equation}\label{eq2**}
\left(\frac{c_{r+3}'}{p-1}\right)\frac{1}{m_1'\cdots m_{r+3}'}=\left(\frac{c_{r+1}'}{p-1}\right)\frac{1}{m_1'\cdots m_{r+1}'}-
\frac{q_{r+2}'}{m_1'\cdots m_{r+2}'}.
\end{equation}
Equation (\ref{eq1**}) implies
\begin{equation}\label{eq3**}
\frac{1}{2^{r+2}}>\left(\frac{c_{r+1}'}{p-1}\right)\frac{1}{m_1'\cdots m_{r+1}'}-\frac{q_{r+2}'}{m_1'\cdots m_{r+2}'}>0.
\end{equation}

Now $J(S_i/R_i)=(x_i^{c_i})$ and $x_0=x_i^{m_1\cdots m_i}$ so
$\omega(J(S_i/R_i))=\frac{c_i}{m_1\cdots m_i}\omega(x_0)$ and thus
by Proposition \ref{Prop100}, (\ref{eq61*}) and (\ref{eq65*}), we have that
$$
-{\rm dist}(\omega/\nu)=\frac{1}{p-1}\inf_i\{\omega(J(S_i/R_i))\}=0.
$$

We  have that 
$J(T_i/S_i)=(z_i^{c_i'})$ and $z_0=z_i^{m_1'\cdots m_i'}$ so
$\omega(J(T_i/S_i))=\frac{c_i'}{m_1'\cdots m_i'}\omega(z_0)$ and thus
by Proposition \ref{Prop100}, (\ref{eq2**}) and (\ref{eq3**}), we have that
$$
-{\rm dist}(\mu/\omega)=\frac{1}{p-1}\inf_i\{\omega(J(T_i/S_i))\}=0.
$$

\end{proof}
%%%%%%%%%%%%%%%%%%

\section{Calculation of distance in some examples}\label{Example}

We give an analysis of the tower of two Artin-Schreier extensions constructed in \cite[Theorem 7.38]{CP}. The example gives a diagram
$$
\begin{array}{ccccc}
R_1&\rightarrow &A_1&\rightarrow & S_1\\
\downarrow&&\downarrow&&\downarrow\\
R_2&\rightarrow &A_2&\rightarrow &S_2\\
\downarrow&&\downarrow&&\downarrow\\
\vdots&&\vdots&&\vdots
\end{array}
$$
where the first two columns and the last two columns are diagrams of the type of (\ref{eq30}). The union of the $S_i$ is the valuation ring of a rational rank 1 nondiscrete valuation $\omega$. The vertical arrows are all standard sequences of quadratic transforms. 

The rows are such that $R_i\rightarrow A_i$ is of type 2 if $i$ is odd and of type 1 if $i$ is even. The extension $A_i\rightarrow S_i$ is  of type 1 if $i$ is odd and of type 2 if $i$ is odd.  $R_i$ has regular parameters $u_i,v_i$ and $S_i$ has regular parameters $x_i,y_i$ such that 
$$
u_i=\gamma_i x_i^p, v_i=\tau_i y_i^p+x_ig_i
$$
for all $i$, where $\gamma_i,\tau_i$ are units in $S_i$ and $g_i\in S_i$. A further analysis in \cite{CP} shows that strong local monomialization fails for this example. 

An example is given in the later paper \cite{C2} where the condition of local monomialization itself fails.

The example of Section 7 of \cite{CP} is a composite of two defect Artin-Schreier extensions,
$$ 
K=k(u,v)\rightarrow K_1=k(x,v)\rightarrow K^*=k(x,y)
$$
where 
\begin{equation}\label{eq77}
u=\frac{x^p}{1-x^{p-1}}, v=y^p-x^cy
\end{equation}
with $c$ a positive integer which is divisible by $p-1$. A rational rank 1 valuation $\omega$ is given of $K^*$, which is trivial on $k$. Let $\nu_1$ be the restriction of $\omega$ to $K_1$ and $\nu$ be the restriction of $\omega$ to $K$.

We will determine the distances of these extensions, illustrating an application of Proposition \ref{Prop100}, showing that the extension $K\rightarrow K_1\rightarrow K^*$ is a tower of two  dependent defect Artin-Schreier extensions.
The first of these extensions was computed by a different method in \cite{EG}.

We have a sequence of algebraic regular local rings of $K$, $K_1$ and $K^*$, 
$$
R_1=k[u,v]_{(u,v)}\rightarrow A_1=k[x,v]_{(x,v)}\rightarrow S_1=k[x,y]_{(x,y)}
$$
such that $\omega$ dominates $S_1$. We normalize $\omega$ by setting $\omega(x)=1$. There are  sequences of homomorphisms
$$
\begin{array}{ccccc}
\vdots&&\vdots&&\vdots\\
\uparrow&&\uparrow&&\uparrow\\
R_{r+1}&\rightarrow &A_{r+1}&\rightarrow &S_{r+1}\\
\uparrow&&\uparrow&&\uparrow\\
R_r&\rightarrow&A_r&\rightarrow &S_r\\
\uparrow&&\uparrow&&\uparrow\\
\vdots&&\vdots&&\vdots\\
\uparrow&&\uparrow&&\uparrow\\
R_2&\rightarrow &A_2&\rightarrow &S_2\\
\uparrow&&\uparrow&&\uparrow\\
R_1&\rightarrow &A_1&\rightarrow &S_1\\
\end{array}
$$
where the vertical arrows are sequences of quadratic transforms which are dominated by $\omega$ and $S_r$ dominates $A_r$ and $A_r$ dominates $R_r$. These sequences are calculated above $R_1$ and $S_1$ in \cite{CP} and above $A_1$ in \cite{C1}.
 The homomorphism $R_k\rightarrow A_k$ is of type 1 if $k$ is even, and of type 2 if $k$ is odd. The homomorphism $A_k\rightarrow S_k$ is of type 1 if $k$ is odd and of type 1 if $k$ is even. 

The local ring $A_r$ has regular parameters $(x_{A_k}, v_{A_k})$ and $(x_{A_k},\overline v_k)$ such that $A_k\rightarrow A_{k+1}$ is defined (equations (70) and (71) of \cite{C1}) by
\begin{equation}\label{eq70}
x_{A_k}=x_{A_{k+1}}^p(v_{A_{k+1}}+1), \overline v_{A_k}=x_{A_{k+1}}
\mbox{ if $k$ is odd}
\end{equation}
and 
\begin{equation}\label{eq71}
x_{A_k}=x_{A_{k+1}}^{p^3}(v_{A_{k+1}}+1), \overline v_{A_k}=x_{A_{k+1}}
\mbox{ if $k$ is even.}
\end{equation}
The local ring $S_r$ has regular parameters $(x_{S_k}, y_{S_k})$ and $(x_{S_k},\overline y_{S_k})$ such that $S_k\rightarrow S_{k+1}$ is defined for all $k\ge 1$ (equation (45) of \cite{C1}) by 
\begin{equation}\label{eq72}
x_{S_k}=x_{S_{k+1}}^{p^2}(\overline y_{k+1}+1), \overline y_k=x_{S_{k+1}}.
\end{equation}
By \cite[Theorem 7.38]{CP}, the local ring $R_r$ has regular parameters $(x_{R_k}, y_{R_k})$ such that 
the homomorphism $R_k\rightarrow S_k$  has a stable form
$$
u_{R_k}=\gamma_kx_{S_k}^p, v_{R_k}=\alpha_ky_{S_k}^p+x_{S_k}g_k
$$
for all $k\ge 1$, where $\gamma_k$ and $\alpha_k$ are units in $S_k$ and $g_k\in S_k$.
We have that $\omega(x_{A_1})=1$ and for $k\ge 2$, we deduce from
(\ref{eq70}) and (\ref{eq71}) that
$$
\omega(x_{A_k})=\left\{\begin{array}{ll}
\frac{1}{p^{2k-2}} &\mbox{ if $k$ is odd,}\\
\frac{1}{p^{2k-3}}&\mbox{ if $k$ is even.}
\end{array}\right.
$$
Letting $J(A_k/R_k)=(x_{A_k}^{c_k})$, we calculate from 2) of Remark \ref{RemarkA} that
if $k$ is even, then
\begin{equation}\label{eq73}
\left(\frac{c_{k+1}}{p-1}\right)\omega(x_{A_{k+1}})=\left(\frac{c_k}{p-1}\right)\omega(x_{A_k})-\omega(\overline v_{A_k})+\omega(x_{A_{k+1}})
=\left(\frac{c_k}{p-1}\right)\omega(x_{A_k})
\end{equation}
and if $k$ is odd, we calculate from 1) of Remark \ref{RemarkB} that
\begin{equation}\label{eq74}
\left(\frac{c_{k+1}}{p-1}\right)\omega(x_{A_{k+1}})=\left(\frac{c_k}{p-1}\right)\omega(x_{A_k})-\omega(x_{A_{k}})
=\left(\frac{c_k}{p-1}\right)\omega(x_{A_k})-\frac{1}{p^{2k-2}}.
\end{equation}
From 
$$
u_{R_1}=\frac{x_{A_1}^p}{1-x_{A_1}^{p-1}}, v_{R_1}=v_{A_1}
$$
 we compute $J(A_1/R_1)=(x_{A_1}^{2p-2})$ so $\omega(J(A_1/R_1)=2(p-1)$ and $c_1=2p-2$.
 Now we compute from equations (\ref{eq73}) and (\ref{eq74}) and Proposition \ref{Prop100} that
 $$
 -{\rm dist}(\nu_1/\nu)=\frac{1}{p-1}\inf_k\omega(J(A_k/R_k))
 =1-\left(\sum_{i=1}^{\infty}\frac{1}{p^{4i}}\right)=1-\left(\frac{1}{p^4-1}\right)=\frac{p^4-2}{p^4-1}.
 $$
 Since ${\rm dist}(\nu_1/\nu)$ is less than zero, the extension is dependent.
This distance is computed using a different method in \cite{EG}. 

From (\ref{eq72}), we compute
\begin{equation}\label{eq72*}
\omega(x_{S_k})=\frac{1}{p^{2k-2}}
\end{equation}
for $k\ge 1$. If $k$ is odd, we compute from 2) of Remark \ref{RemarkA} that
\begin{equation}\label{eq75}
\left(\frac{c_{k+1}}{p-1}\right)\omega(x_{S_{k+1}})=\left(\frac{c_k}{p-1}\right)\omega(x_{S_k})-\omega(\overline y_{S_k})+\omega(x_{S_{k+1}})
=\left(\frac{c_k}{p-1}\right)\omega(x_{S_k}).
\end{equation}
 If $k$ is even, we compute from 1) of Remark \ref{RemarkB} that 
\begin{equation}\label{eq76}
\left(\frac{c_{k+1}}{p-1}\right)\omega(x_{S_{k+1}})=\left(\frac{c_k}{p-1}\right)\omega(x_{S_k})-\omega(x_{S_{k}})
=\left(\frac{c_k}{p-1}\right)\omega(x_{S_k})-\frac{1}{p^{2k-2}}.
\end{equation}
Now we compute from (\ref{eq77}) that $J(S_1/A_1)=(x_1^c)$ so $c_1=c$. 
Now we compute from equations (\ref{eq75}) and (\ref{eq76}) and Proposition \ref{Prop100} that
 $$
 \begin{array}{lll}
 -{\rm dist}(\omega/\nu_1)&=&\frac{1}{p-1}\inf_k\omega(J(S_k/A_k))\\
 &=&\frac{c}{p-1}-\frac{1}{p^2}\left(\sum_{i=0}^{\infty}\frac{1}{p^{4i}}\right)\\
 &=&\frac{c}{p-1}-\frac{1}{p^2}\left(\frac{p^4}{p^4-1}\right)=\frac{cp^3+(c-1)p^2+cp+c}{p^4-1}.
 \end{array}
 $$
 Now $cp^3+(c-1)p^2+cp+c=0$ is positive for all positive integers $c$,  so $\mbox{dist}(\omega/\nu_1)$ is less than zero for all  positive integral $c$ and thus the extension is dependent.

\section{Appendix} In this appendix we give proofs of the results on defect cuts and ramification cuts
of Artin-Schreier extensions from \cite{KP}. These results are stated in \cite{KP} and  their proofs are outlined there. 
%The proofs given here are  not completely from \cite{KP}, and any defect in them is my responsibility. 

We suppose throughout this section that $L$ is an Artin-Schreier extension of a field $K$ of characteristic $p$, $\omega$ is a  rank 1 valuation of $L$ and $\nu$ is the restriction of $\omega$ to $K$. We suppose that $L$ is a defect extension of $K$.  We will use the notation of Subsections \ref{Galois} and \ref{Rank1AS}.

% To simplify notation, we suppose that we have an embedding of $\nu L$ in $\RR$. Let $\mbox{Gal}(L/K)$ be the Galois group of $L/K$.

%The defect $\delta(\omega/\nu)$ can be defined by the equation (c.f. \cite{Ku} or Section 7.1 \cite{CP})
%$$
%|G^s(\omega/\nu)|=[L\omega:K\nu][\omega L:\nu K]\delta(\omega/\nu),
%$$
%where $G^s(\omega/\nu)$ is the decomposition group of $L/K$. $\delta(\omega/\nu)$ is always a power of the characteristic of $K\nu$.
%The assumption that $L$ is a defect Artin Schreier extension of $K$ implies that $G^s(\omega/\nu)=\mbox{Gal}(L/K)$, so that $\omega$ is the unique extension of $\nu$ to $L$ (Extension Theorem, page 181 \cite{Neu}), and that $L$ is an immediate extension of $K$  ($L\omega=K\nu$ and $\omega L =\nu K$).

 Let $\Theta$ be an Artin-Schreier generator of $K$.
 % that is, there is an expression
%$$
%\Theta^p-\Theta = a
%$$
%for some $a\in K$. 
We have that 
$$
\mbox{Gal}(L/K)\cong \ZZ_p=\{{\rm id}, \sigma_1,\ldots,\sigma_{p-1}\},
$$
where $\sigma_i(\Theta)=\Theta+i$.
Since $L/K$ is an immediate extension, the set $\omega(\Theta - K)$ is an initial segment in $\nu K$ which  has no maximal element. Further, $\omega(\Theta)<0$ by \cite[Lemma 2.28]{Ku}.
%We define a cut 
%$d(\omega/\nu)$ 
%i%n $\RR$ by extending the cut $\mbox{dist}(\Theta,K)$ in the divisible hull $\widetilde{\omega K}$ of $\omega K$ to a cut of $\RR$ by taking the initial segment of 
%$d(\omega/\nu)$
% the extended cut to be the least initial segment of $\RR$ in which  the cut $\mbox{dist}(\Theta,K)$ is cofinal.
 %This cut is then $\mbox{dist}(\Theta,K)\uparrow \RR$.  This cut is either $s$ or $s^{-}$ for some $s\in \RR$. We have that $\mbox{dist}(\Theta,K)\uparrow \RR=s^{-}$ where
%$s$ is a non positive real number by Corollary 2.30 \cite{Ku} since $L/K$ is a defect extension. We will set 
Let $s=\mbox{dist}(\omega/\nu)\in \RR$, so that 
$$
\mbox{dist}(\Theta,K)\uparrow \RR=s^{-}=\mbox{dist}(\omega/\nu)^{-}\le 0^-.
$$
%The real number $\mbox{dist}(\omega/\nu)$  is well defined since it is independent of choice of Artin Schreier generator of $L/K$ by Lemma 4.1 \cite{Ku}.

There exists a sequence $\{c_i\}_{i\in \NN}$ in $K$ such that 
$$
\omega(\Theta-c_i)<\omega(\Theta-c_{i+1})
$$
for all $i$, and
$$
\lim_{i\rightarrow \infty}\omega(\Theta-c_i)=\mbox{dist}(\omega/\nu).
$$

Suppose that $r<s$. Then
$$
\nu(c_s-c_r)=\omega((\Theta-c_r)-(\Theta-c_s))=\omega(\Theta-c_r).
$$
Thus for $r<s<t$,
$$
\nu(c_s-c_r)=\omega(\Theta-c_r)<\omega(\Theta-c_s)=\nu(c_t-c_s),
$$
so $\{c_i\}$ is a pseudo convergent sequence in $K$ (\cite{Kap} or Page 39 \cite{Sch}). For $r\in \NN$, define 
$$
\gamma_r=\{\nu(c_s-c_r)\mid r<s\}=\omega(\Theta-c_r).
$$
The Artin-Schreier generator $\Theta$ is a pseudo limit of $\{c_i\}$ (\cite{Kap} or page 47 \cite{Sch}). 

\begin{Lemma}
 The pseudo convergent sequence $\{c_i\}$ does not have a pseudo limit in $K$.
 \end{Lemma}
 
 \begin{proof} Suppose $c\in K$ is a pseudo limit of $\{c_i\}$ in $K$. Then
 $$
 \nu(c-c_r)=\gamma_r=\omega(\Theta-c_r)\mbox{ for all $r$}.
 $$
 Thus for all $r$,
 $$
 \omega(\Theta-c)=\omega((\Theta-c_r)+(c_r-c))\ge \omega(\Theta-c_r)
 $$
 so $\omega(\Theta-c)\ge \mbox{dist}(\omega/\nu)$, a contradiction.
 \end{proof}
 
 \begin{Lemma}\label{Lemma10} Suppose that $f(x)\in K[x]$ is a polynomial such that $\mbox{deg}(f)<p$. Then there exists
 $t_0\in \NN$ such that $\nu(f(c_t))=\nu(f(c_{t_0}))$ for $t\ge t_0$.
 \end{Lemma}

 \begin{proof} Since $\nu$ has rank 1, by \cite[Lemma 10]{Kap}, the smallest degree of a  polynomial  $g(x)\in K[x]$ such that $\nu(g(c_t))$ does not stabilize for large $t$ is a power of $p$. Since $\{c_t\}$ does not have a pseudo limit in $K$, we have that this degree is $\ge p$. 
 \end{proof}

 \begin{Proposition}\label{Prop11} Suppose that $f(x)\in K[x]$ is a polynomial of degree $r<p$. Let $\Theta_i=\Theta-c_i$. Then there are polynomials $g_j(Y)\in K[Y]$ of degree $\le j$ such that
 \begin{equation}\label{eq13*}
 f(\Theta)=g_0(c_i)\Theta_i^r+g_1(c_i)\Theta_i^{r-1}+\cdots+g_{r-1}(c_i)\Theta_i+g_r(c_i)
 \end{equation}
 for all $i$, with $g_0(c_i)=g_0$ a non zero element of $K$. Further, there exists $i_0$  and $\lambda_j\in \nu K$ (depending on $f$) such that 
 $$
 \nu(g_j(c_i))=\lambda_j
 $$
 for $i\ge i_0$ and $0\le j\le r$, and
\begin{equation}\label{eq12*}
 \omega(g_j(c_i)\Theta_i^{r-j})\ne \omega(g_k(c_i)\Theta_i^{r-k})
 \end{equation}
 for all $0\le j<k\le r$  and $i\ge i_0$.
 \end{Proposition}
 
 \begin{proof} We have a factorization
 $f(x)=f_0f_1(x)\cdots f_l(x)$ where $f_0\in K$ and $f_i(x)$ are monic and irreducible for $1\le i\le l$. Let $r_j=\mbox{deg}(f_j(x))$.
 Let $\Omega$ be an algebraic closure of $K$ containing $L$. We have factorizations
 $$
 f_j(x)=(x-a_{j1})(x-a_{j2})\cdots (x-a_{jr_j})
 $$
 with $a_{ji}\in \Omega$ for $1\le j\le r$, giving expressions
 $$
 f_j(x)=x^{r_j}-S_1(a_{j1},\ldots,a_{jr_j})x^{r_j-1}+\cdots +(-1)^{r_j}S_{r_j}(a_{j1},\ldots,a_{jr_j})
 $$
 where $S_i$ is the elementary symmetric function of degree $i$. Let $y$ be an indeterminate. Then
 $$
 f_j(x+y)=x^{r_j}-S_1(a_{j1}-y,\ldots,a_{jr_j}-y)x^{r_j-1}+\cdots +(-1)^{r_j}S_{r_j}(a_{j1}-y,\ldots,a_{jr_j}-y).
 $$
 Let $L_j=K(a_{j1},\ldots,a_{jr_j})$ and set 
 $$
 h_i=S_i(a_{j1}-y,\ldots,a_{jr_j}-y)\in L_j[y]
 $$
 for $1\le i\le r_j$. $h_i$ is a polynomial of degree $i$. $h_i$ is invariant under permutation of the $a_{jk}$ and 
 $L_j$ is Galois over $K$ (since it is a normal extension of $K$ and $r_j<p)$. Thus $h_i\in K[y]$ for $i\le r_j$, and we have an expression
 \begin{equation}\label{eq3}
 f(x+y)=g_0x^r+g_1(y)x^{r-1}+\cdots+g_{r-1}(y)x+g_r(y)
 \end{equation}
 where $g_i(y)\in K[y]$ is a polynomial of degree $\le i$ ($g_0=f_0$). For $i\in\NN$, we have an expression
 $$
 f(\Theta)=f(\Theta_i+c_i)
 =g_0\Theta_i^r+g_1(c_i)\Theta_i^{r-1}+\cdots+g_{r-1}(c_i)\Theta_i+g_r(c_i).
 $$
 By Lemma \ref{Lemma10}, there exists $i_0$ such that  $\nu(g_j(c_i))$ is a  constant value $\lambda_j$ for $i\ge i_0$ and $0\le j\le r$. Now $\omega(\Theta_{i+1})>\omega(\Theta_i)$ for all $i$, so for all $j,k$ with $0\le j< k\le r$, there exists $i(j,k)$ such that 
 $$
 \omega(\Theta_i)\ne \frac{\lambda_k-\lambda_j}{t}
 $$
 for $t$ any integer with $1\le t\le r$, whenever  $i>i(j,k)$.
 
 Thus for $i$ such that $i>i_0$ and $i>\max\{i(j,k)\mid 0\le j < k \le r\}$ and $0\le j<k\le r$,
 $$
 \omega(g_j(c_i)\Theta_i^{r-j})\ne \omega(g_k(c_i)\Theta_i^{r-k}).
 $$
 \end{proof}

 \begin{Corollary}\label{Cor10} The valuation ring $\mathcal O_{\omega}$ is generated as an $\mathcal O_{\nu}$-module  by 
 \begin{equation}\label{eqN30}
 \{g\Theta_i^j\mid g\in K, 0\le j\le p-1, i\in N\mbox{ and }\omega(g\Theta_i^j)\ge 0\}.
 \end{equation}
 \end{Corollary}
 
 \begin{proof} Let $M$ be the $\mathcal O_{\nu}$-module generated by the set (\ref{eqN30}). The module $M$ is certainly contained in $\mathcal O_{\omega}$. Suppose that $h\in \mathcal O_{\omega}$. 
 Then $\omega(h)\ge 0$ and $h=f(\Theta)$ for some polynomial $f\in K[x]$ of degree $<p$. By Proposition \ref{Prop11}, we have an expression (\ref{eq13*}) of $h$. Taking $i$ sufficiently large (so that (\ref{eq12*}) holds) we have that 
 \begin{equation}\label{eq14*}
 0\le \omega(h)=\min\{\omega(g_j(c_i)\Theta_i^{r-j})\mid 0\le j\le r\}.
 \end{equation}
 Thus $h\in M$.
 \end{proof}
 
 \begin{Corollary}\label{Cor11} Suppose that $h\in L$. Then 
 $$
 \omega(\sigma(h)-h)=\omega(\tau(h)-h)
 $$
 for $\sigma, \tau\in \mbox{Gal}(L/K)$ which are both not the identity.
 \end{Corollary}
 
 \begin{proof} $h=h(\Theta)$ has an expression of the form (\ref{eq13*}) of Proposition \ref{Prop11} such that (\ref{eq12*}) holds.
 We compute using (\ref{eq3}) for $0<j<p$,
 $$
 \begin{array}{lll}
 \sigma_j(h)-h&=&h(\Theta+j)-h(\Theta)=h(\Theta_i+j+c_i)-h(\Theta_i+c_i)\\
 &=&g_0(c_i)(\Theta_i+j)^r+\cdots+g_{r-1}(c_i)(\Theta_i+j)+g_r(c_i)\\
 &&-[g_0(c_i)\Theta_i^r+\cdots +g_{r-1}(c_i)\Theta_i+g_r(c_i)]\\
 &=&g_0(c_i)\left(\sum_{k=0}^{r-1}\binom{r}{k}j^{r-k}\Theta_i^k\right)
 +g_1(c_i)\left(\sum_{k=0}^{r-2}\binom{r-1}{k}j^{r-1-k}\Theta_i^k\right)
 +\cdots+g_{r-1}(c_i)j.
 \end{array}
 $$
 Since $\omega(\Theta_i)<0$, we have that
 $$
 \omega\left(g_l(c_i)\left(\sum_{k=0}^{r-l-1}\binom{r-l}{k}j^{r-l-k}\Theta_i^k\right)\right)=\omega(g_l(c_i)j\Theta_i^{r-l-1})
 =\omega(g_l(c_i)\Theta_i^{r-l-1})
$$
 for $0\le l\le r-1$. Thus
 \begin{equation}\label{eq15}
 \omega(\sigma_j(h)-h)=\min\{\omega(g_l(c_i)\Theta_i^{r-l-1})\mid 0\le l\le r-1\}
 \end{equation} 
 for $i\gg 0$ by equation (\ref{eq12*}).
 \end{proof}
 
 For $\alpha\in \omega L$, define the ideal $I_{\alpha}=\{f\in \mathcal O_L\mid \omega(f)\ge \alpha\}$ in $\mathcal O_{\omega}$.
 Then $G_{I_{\alpha}}$ (defined in Subsection \ref{SubsecHRG}) is  the subgroup     $$
  G_{I_{\alpha}}=\{s\in \mbox{Gal}(L/K)\mid \omega(s(x)-x)\ge \alpha\mbox{ for all }x\in \mathcal O_{\omega}\}
  $$
 of $G=\mbox{Gal}(L/K)$.
  
 \begin{Corollary}\label{Cor12} Continuing with our assumption that $L/K$ is a defect Artin-Schreier extension and $\omega$ has rank 1, Suppose that $\alpha\in \RR$. Then 
$$
G_{I_{\alpha}}=\left\{\begin{array}{ll}
\{1\}&\mbox{ if }\alpha>-\mbox{dist}(\omega/\nu)\\
\mbox{Gal}(L/K)&\,\mbox{ if } \alpha\le -\mbox{dist}(\omega/\nu).
\end{array}\right.
$$
\end{Corollary}

\begin{proof} Suppose that $h\in \mathcal O_{\omega}$. Then $h$ has an expression of the form (\ref{eq13*}) of Proposition \ref{Prop11} such that (\ref{eq12*}) holds. By the calculation of (\ref{eq14*}) of the proof of Corollary \ref{Cor10} we have that for $i\gg 0$,
\begin{equation}\label{eq16}
 0\le \omega(h)=\min\{\omega(g_j(c_i)\Theta_i^{r-j})\mid 0\le j\le r\}.
 \end{equation}
 By the calculation of (\ref{eq15}) of the proof of Corollary \ref{Cor11} and (\ref{eq16}),  we have that for $\mbox{id}\ne \sigma\in \mbox{Gal}(L/K)$ and $i\gg 0$,
$$
 \omega(\sigma(h)-h)=\min\{\omega(g_l(c_i)\Theta_i^{r-l-1})\mid 0\le l\le r-1\}\ge -\omega(\Theta_i)> -\mbox{dist}(\omega/\nu).
 $$
 Thus $G_{\alpha}=\mbox{Gal}(L/K)$ for $\alpha\le -\mbox{dist}(\omega/\nu)$.
 
 Given $\epsilon>0$ there exists $c\in K$ such that $\omega(\Theta-c)>\mbox{dist}(\omega/\nu)-\frac{\epsilon}{2}$. Let $\overline{\Theta}=\Theta-c$. The group $\omega K$
  is dense in $\RR$, so there exists $g\in K$ such that 
 $$
 0\le \omega(g\overline\Theta)<\frac{\epsilon}{2}.
 $$
 We have that $\sigma_1(g\overline\Theta)-g\overline \Theta=g$, so
 $$
 \omega(\sigma_1(g\overline\Theta)-g\overline\Theta)=\omega(g)<-\mbox{dist}(\omega/\nu)+\epsilon.
 $$
 Thus $G_{\alpha}=\{1\}$ for $\alpha>-\mbox{dist}(\omega/\nu)$.
\end{proof}

As a consequence of the above corollary, we obtain the following theorem.

\begin{Theorem}\label{ramcut}(Kuhlmann and Piltant \cite{KP}) Continuing with our assumption that $L/K$ is a defect Artin-Schreier extension and $\omega$ has rank 1, let ${\rm Ram}(\omega/\nu)$ be the ramification cut of $L/K$ defined in subsection \ref{Galois}. Then 
$$
{\rm dist}(\omega/\nu)^{-}\cap \nu K={\rm dist}(\Theta,K)\cap \nu K=-{\rm Ram}(\omega/\nu)\cap \nu K.
$$
\end{Theorem}

\begin{proof} By Corollary \ref{Cor12}
$$
({\rm Ram}(\omega/\nu)\uparrow \RR)^R=\{\alpha\in \RR\mid 
G_{I_{\alpha}}=1\}=\cup\{\alpha\in \RR\mid \alpha>-\mbox{dist}(\omega/\nu)\}.
$$
Thus 
$$
(-{\rm Ram}(\omega/\nu)\uparrow\RR)^L=\{\alpha\in\RR\mid \alpha<\mbox{dist}(\omega/\nu)\}
$$
and so $\mbox{dist}(\Theta,K)\cap \nu K=-{\rm Ram}(\omega/\nu)\cap \nu K$.
\end{proof}

\begin{Lemma}\label{Lemma20}(Kuhlmann and Piltant, \cite{KP})
Suppose that $K$ and $L$ are two dimensional algebraic function fields over an algebraically closed field $k$ of characteristic $p>0$ and $K\rightarrow L$ is an Artin-Schreier extension. Let $\omega$ be a rational rank one nondiscrete valuation of $L$ and let $\nu$ be the restriction of $\omega$ to $K$. Suppose that $L$ is a defect extension of $K$.

Suppose that $R$ is a regular algebraic local ring of $K$ and $S$ is a regular algebraic local ring of $L$ such that $\omega$ dominates $S$, $S$ dominates $R$ and $R\rightarrow S$ is of type 1 or 2.
Inductively applying Theorems \ref{TheoremA} and  \ref{TheoremB}, we construct a diagram where the horizontal sequences are birational extensions of regular local rings
\begin{equation}\label{eq23*}
\begin{array}{ccccccc}
S=S_0&\rightarrow &S_1&\rightarrow & S_2&\rightarrow & \cdots\\
\uparrow&&\uparrow&&\uparrow&&\\
R=R_0&\rightarrow &R_1&\rightarrow & R_2&\rightarrow & \cdots
\end{array}
\end{equation}
with $\cup_{i=1}^{\infty}S_i=\mathcal O_{\omega}$. 
Further assume that for each map $R_i\rightarrow S_i$, there are regular parameters $u,v$ in $R_i$ and $x,y$ in $S_i$ such that one of the following forms hold:
\begin{equation}\label{eq20*}
u=x, v =f
\end{equation}
where $\dim_kS_i/(x,f)=p$, or
\begin{equation}\label{eq21*}
u=\delta x^p, v = y
\end{equation}
where $\delta$ is a unit in $S_i$ and  in both cases that  $x=0$ is a local equation of the critical locus of $\mbox{Spec}(S_i)\rightarrow \mbox{Spec}(R_i)$.

 Let
 $$
 J_i=J(S_i/R_i)=(\frac{\partial u}{\partial x}\frac{\partial v}{\partial y}-\frac{\partial u}{\partial y}\frac{\partial v}{\partial x})
 $$
 be the Jacobian ideal of the map $R_i\rightarrow S_i$.

 Then there exists $c>0$ such that $J_i =x^cS_i$ (since the critical locus is supported on $x=0$).
  Let $\sigma$ be a generator of $\mbox{Gal}(L/K)$. Then 
  \begin{enumerate}
  \item[1)] $\omega(\sigma(y)-y)=\frac{c}{p-1}\omega(x)=\frac{1}{p-1}\omega(J_i)$ if (\ref{eq20*}) holds,
  \item[2)] $\omega(\sigma(x)-x)=\frac{c}{p-1}\nu(x)=\frac{1}{p-1}\omega(J_i)$ if (\ref{eq21*}) holds.
  \end{enumerate}
  \end{Lemma}

\begin{proof} We  prove the first statement 1). The proof of the second statement is similar. Suppose that a form (\ref{eq20*}) holds.
 Let $N$ be the $S_i$-ideal $N=\mbox{Ann}_{S_i}(\Omega^1_{S_i/R_i})$.

Since $\omega$ is the unique extension of $\nu$ to $L$ and $R_i\rightarrow S_i$ is quasi finite with complexity $p=[L:K]$, we have that $S_i$ is the integral closure of $R_i$ in $L$ and is thus a finite $R_i$-module.
 There exists a unit $\delta\in S_i$ and $\gamma\in S_i$ such that 
 $$
 u=x, v=\delta y^p+x\gamma.
 $$
 Let $M$ be the $R_i$-module $M=R_i+R_iy+\cdots+R_iy^{p-1}$. We have that
 $$
 y^p=\delta^{-1}v-u\delta^{-1}\gamma\in (u,v)S_i
 $$
 and $x=u\in (u,v)S_i$ 
 so $S_i=M+(u,v)S_i$. Thus $S_i=M$ by Nakayama's lemma. Let $f(t)\in K[t]$ be the minimal polynomial of $y$ over $K$. 
 The polynomial $f(t)$ has degree $p$ since $[L:K]=p$. 
 Since $R_i$ is normal and $y$ is integral over $R_i$, by Theorem 4 on page 260 of \cite{ZS1}, the coefficients of  $f(t)$ are in $R_i$, and thus $S_i\cong R_i[t]/(f(t))$.  We  have an isomorphism of $S_i$-modules
 $$
 \Omega^1_{S_i/R_i}\cong S_i/f'(y)S_i
 $$
 where $f'(t)=\frac{df}{dt}$. Thus  $N=(f'(y))$. We compute $N$ in another way, from the right exact sequence 
 $$
 \Omega^1_{R_i/k}\otimes_{R_i}S_i\rightarrow \Omega^1_{S_i/k}\rightarrow \Omega^1_{S_i/R_i}\rightarrow 0,
 $$
 showing that we have a presentation
 $$
 S_i^2\stackrel{A}{\rightarrow}S_i^2\rightarrow \Omega^1_{S_i/R_i}\rightarrow 0,
$$
where
$$
A=\left(\begin{array}{cc} \frac{\partial u}{\partial x}& \frac{\partial u}{\partial y}\\
\frac{\partial v}{\partial x}& \frac{\partial v}{\partial y}
\end{array}\right)
=\left(\begin{array}{cc}
1&0\\
\frac{\partial v}{\partial x}& \frac{\partial v}{\partial y}
\end{array}\right).
 $$
 From the fact that 
 $$
 \left\{ \left(\begin{array}{l} 1\\\frac{\partial v}{\partial x}\end{array}\right), \left(\begin{array}{l}0\\1\end{array}\right)\right\}
 $$
 is an $S_i$-basis of $S_i^2$, 
 we see that $\Omega^1_{S_i/R_i}\cong S_i/\frac{\partial v}{\partial y}S_i$, and so $N=(\frac{\partial v}{\partial y})=J_i=(x^c)$.
 Factoring 
 $$
 f(t)=\prod_{\tau\in \mbox{Gal}(L/K)}(t-\tau(y))
 $$
  in $L[t]$, we see that 
 $$
 f'(y)=\prod_{id\ne \tau\in \mbox{Gal}(L/K)}(y-\tau(y)).
 $$
 Thus 
 $$
 c\omega(x)=\sum_{id\ne \tau\in \mbox{Gal}(L/K)}\omega(\tau(y)-y)=(p-1)\omega(\sigma(y)-y)
 $$
 by Corollary \ref{Cor11}.
 \end{proof}

 \begin{Proposition}\label{Prop100} (Kuhlmann and Piltant, \cite{KP}) Let assumptions be as in Lemma \ref{Lemma20}. Then the distance $\mbox{dist}(\omega/\nu)$ is computed by the formula
 $$
 -{\rm dist}(\omega/\nu)=\frac{1}{p-1}\inf_i\{\omega(J(S_i/R_i))\}   
 $$
 where the infimum is over the $R_i\rightarrow S_i$ in the sequence  (\ref{eq23*}).
\end{Proposition}

\begin{proof}  Let $\sigma$ be a generator of $\mbox{Gal}(L/K)$. Since $\cup S_i=\mathcal O_{\omega}$, we have that  
$$
-{\rm dist}(\omega/\nu)=\inf\{\omega(\sigma(h)-h)\mid h\in S_i\mbox{ for some $i$}\}
$$
by Corollaries \ref{Cor12} and \ref{Cor11}. By the proof of Lemma \ref{Lemma20}, $h\in S_i$ implies there exists a polynomial $f(t)\in R_i[t]$
of degree $<p$ such that $h=f(z_i)$, where $z_i=y$ if $R_i\rightarrow S_i$ is in case (\ref{eq20*}), $z_i=x$ if $R_i\rightarrow S_i$ is in case (\ref{eq21*}). Thus we have an expression
$$
h=a_0z_i^{p-1}+a_1z_i^{p-2}+\cdots+a_{p-1}
$$
with $a_0,\ldots,a_{p-1}\in R_i$. For $s\ge 1$, we have a factorization
$$
\sigma(z_i)^s-z_i^s=(\sigma(z_i)-z_i)(\sigma(z_i)^{s-1}+z_i\sigma(z_i)^{s-2}+\cdots+z_i^{s-1}).
$$
The valuation $\omega$ is the unique extension of $\nu$ to $L$, so $\omega(\sigma(z_i))=\omega(z_i)\ge 0$. Thus
$\omega(\sigma(z_i)^s-z_i^s)\ge \omega(\sigma(z_i)-z_i)$ for all $s\ge 1$. We have that
$$
\sigma(h)-h=a_0(\sigma(z_i)^{p-1}-z_i^{p-1})+\cdots+a_{p-2}(\sigma(z_i)-z_i)
$$
so $\omega(\sigma(h)-h)\ge \omega(\sigma(z_i)-z_i)$. Thus
$$
-{\rm dist}(\omega/\nu)=\inf\{\omega(\sigma(z_i)-z_i)\}.
$$
The proposition now follows from Lemma \ref{Lemma20}.
\end{proof}


\begin{thebibliography}{1000000000}
\bibitem{Ab1} S. Abhyankar, On the valuations centered in a local domain, Amer. J. Math. 78 (1956), 321 - 348.
\bibitem{Ab2} S. Abhyankar, Local unifiormization on algebraic surfaces over ground fields of characteristic  $p\ne 0$, Annals Math. 63 (1956), 491 - 526. 

\bibitem{Ab3} S.  Abhyankar, Resolution of singularities of embedded algebraic surfaces, second edition, Springer Verlag, New York, Berlin, Heidelberg, 1998.
\bibitem{AKMW} D. Abramovich, K. Karu, K. Matsuki and J. Wlodarczyk, Torification and factorization of birational maps, J. Amer. Math. Soc, 15 (2002), 531 - 572.


\bibitem{BK} A. Blaszczok and F-V. Kuhlmann, Counting the number of distinct distances of elements in valued field extensions, Journal of Algebra 509 (2018), 192 - 211.

\bibitem{CJS} V. Cossart, U. Jannsen, S. Saito, Desingularization: Invariants and strategy, Lecture Notes in Mathematics 2270, Springer, 2020.

\bibitem{CoP} V. Cossart, O. Piltant, Resolution of singularities of arithmetical threefolds, I and II, Journal of algebra 529 (2019), 268 - 535.


\bibitem{C} S.D. Cutkosky, Local factorization and monomialization of morphisms, Ast\'erisque
260, 1999.


\bibitem{C3}  S.D. Cutkosky, A Simpler Proof of Toroidalization of Morphisms from
3-folds to Surfaces, Annales de L'Institut Fourier 63 (2013), 865 - 922.


\bibitem{C4}  S.D. Cutkosky, Toroidalization of dominant morphisms of  3-folds,   Memoirs of the
AMS,   Volume 190, American Mathematical Society, Providence (2007),
222 pages.


\bibitem{C5} S.D. Cutkosky,  Monomialization of Morphisms from 3 Folds to Surfaces, 
Springer Lecture notes,  LNM 1786 (2002), 235 pages.


\bibitem{C7} S.D. Cutkosky, Local Monomialization of Transcendental Extensions, Annales de L'Institut Fourier 55 (2005), 1517 -- 1586.



\bibitem{C6}  S.D. Cutkosky, Resolution of Singularities for 3-folds in positive characteristic, \linebreak  American J. of  Math. 131 (2009), 59-127.



\bibitem{C1} S.D. Cutkosky, Ramification of valuations and local rings in positive characteristic,  Comm. Algebra 44 (2016), 2826 - 2866.

\bibitem{C8} S.D. Cutkosky, the role of defect and splitting in finite generation of extensions of associated graded rings along a valuation, Algebra and Number Theory 11 (2017), 1461 - 1488.

\bibitem{C2} S.D. Cutkosky, Counterexamples to local monomialization in positive characteristic, Math. Annalen 362 (2015), 321 -334.\bibitem{RS} S.D. Cutkosky, Resolution of singularities, Graduate Studies in Mathematics, Vol 63, American Math. Soc., 2004.
%\bibitem{CG} S.D. Cutkosky and L. Ghezzi,  Completions of valuation rings,
%Contemp. math. 386 (2005), 13 - 34.

\bibitem{CMT} S.D. Cutkosky, H. Mourtada and B. Teissier, On the construction of valuations and generating sequences, Algebraic Geometry, Foundation Compositio Mathematica, 8 (2021), 705 - 748.



%\bibitem{CP1} S.D. Cutkosky and O. Piltant,  Monomial resolutions of morphisms of algebraic surfaces,
%%(with  O. Piltant), Communications in Algebra 28 (Special Issue in honor of Hartshorne)
%(2000), 5935-5960.




\bibitem{CP} S.D. Cutkosky and O. Piltant, Ramification of Valuations, Advances in Math. 183 (2004), 1-79.

\bibitem{CV} S.D. Cutkosky and Pham An Vinh, Valuation semigroups of two dimensional  local rings, Proceedings of the London Mathematical Society, 2013.

\bibitem{End} O. Endler, Valuation Theory, Springer Verlag, 1972.

\bibitem{H1} H. Hironaka, Resolution of singularities of an algebraic variety over a field of
characteristic zero, Annals of Math, 79 (1964), 109--326.
\bibitem{EG}  S. El Hitti and L. Ghezzi, Dependent Artin-Schreier defect extensions and strong monomialization, J. Pure Appl. Algebra 220 (2016), 1331 -1342.
%\bibitem{HOST} F.J. Herrera Govantes, M.S. Olalla Acosta, M. Spivakovsky, B. Teissier, 
%Extending a valuation centered in a local domain to the formal completion,
%arXiv:1007.4656.
\bibitem{Kap} I. Kaplansky, Maximal fields with valuations, Duke J. 9 (1942),303-321.
\bibitem{K0} F.-V. Kuhlmann, Valuation theoretic and model theoretic aspects of local uniformization in Resolution of Singularities - A Research Textbook in Tribute to Oscar Zariski. Herwig Hauser, Joseph Lipman, Frans Oort, Adolfo Quiros eds. Progress in Mathematics, vol 181, Birkhauser Verlag Basel (2000), 381 - 456.
\bibitem{Ku} F.-V. Kuhlmann, A classification of Artin Schreier defect extensions and a characterization of defectless fields,
Illinois J. Math. 54 (2010), 397 - 448.
\bibitem{KP} F-V. Kuhlmann and O. Piltant, Higher ramification groups for Artin-Schreier defect extensions, manuscript, 2012.

\bibitem{Li} J. Lipman, Desingularization of 2-dimensional schemes, Annals of Math. 107 (1978), 115 - 207.


%\bibitem{Neu} J. Neukirch, Algebraic Number Theory, Springer, 1999.

\bibitem{NS} J. Novacoski and M. Spivakovsky, Key polynomials and pseudo convergent sequences, Journal of Algebra 495 (2018), 199 - 219.


\bibitem{Pi} O. Piltant, On the Jung method in positive characteristic. Proceedings of the International Conference in Honor of Fr\'ed\'eric Pham (Nice, 2002). Ann. Inst. Fourier (Grenoble) 53 (2003), no. 4, 1237 - 1258.
\bibitem{Sch} O.F.G. Schilling, The Theory of Valuations, American Math. Soc., 1958.
\bibitem{Sp} M. Spivakovsky, Valuations in function fields of surfaces, Amer. J. Math. 112 (1990), 107 - 156.
\bibitem{Z} O. Zariski, The reduction of singularities of an algebraic surface, Annals of Math. 40 (1939), 639 - 689.

\bibitem{Va} M. Vaqui\'e, Famille admissible de valuations et d\'efaut d'une extension, J. Algebra 331 (2007), 859 - 876.


\bibitem{Z2} O. Zariski,  Local uniformization on algebraic varieties, Ann. of Math. 41 (1940), 852 - 896.

\bibitem{Z3} O. Zariski,  Reduction of singularities of algebraic three dimensional varieties, Ann. of Math. 45 (1944), 472 - 542.

\bibitem{ZS1} O. Zariski and P. Samuel, Commutative Algebra Volume I, Van Nostrand, 1958.
\bibitem{ZS2} O. Zariski and P. Samuel, Commutative Algebra Volume II, Van Nostrand, 1960.
\end{thebibliography}
\end{document}